\documentclass{amsart}
\usepackage{amsmath, amsthm, amssymb,color}
\usepackage{url}
\usepackage{hyperref}

\newtheorem{theorem}{Theorem}
\newtheorem{proposition}[theorem]{Proposition}

\newtheorem{corollary}[theorem]{Corollary}
\newtheorem{lemma}[theorem]{Lemma}

\theoremstyle{definition}

\numberwithin{theorem}{section}
\numberwithin{equation}{section}

\newcommand{\F}{\mathbb{F}}
\newcommand{\FF}{\mathbb{F}}
\newcommand{\QQ}{\mathbb{Q}}

\newcommand{\USp}{\mathrm{USp}}
\newcommand{\U}{\mathrm{U}}

\newcommand{\M}{\mathrm{M}}

\newcommand{\dd}{\;\mathrm{d}}

\newcommand{\divides}{\mathbin{|}}

\DeclareMathOperator{\Tr}{tr}
\DeclareMathOperator{\re}{Re}

\title[Traces, high powers and one level density]{Traces, high powers and one level density for families of curves over finite fields}
\author[A. Bucur]{Alina Bucur}
\address[Alina Bucur]{Department of Mathematics\\
University of California, San Diego\\
9500 Gilman Drive \#0112\\
La Jolla, CA 92093\\
U.S.A}
\thanks{The first author was partially supported by Simons Foundation grant \#244988.}
\email{alina@math.ucsd.edu}

\author[E. Costa]{Edgar Costa}
\address[Edgar Costa]{Department of Mathematics \\
Dartmouth College \\
27 N Main Street \\
6188 Kemeny Hall\\
Hanover, NH 03755-3551\\
U.S.A}
\thanks{The second author was partially supported by FCT doctoral grant SFRH/BD/69914/2010.}
\email{edgarcosta@math.dartmouth.edu}

\author[C. David]{Chantal David}
\address[Chantal David]{
    Concordia University\\
    1455 de Maisonneuve West\\
    Montr\'{e}al, Qu{e}bec\\
Canada H3G 1M8}
\email{cdavid@mathstat.concordia.ca}
\thanks{The third author was partially supported by a NSERC Discovery Grant 155635-2013.}

\author[J. Guerreiro]{Jo\~{a}o Guerreiro}
\address[Jo\~{a}o Guerreiro]{Department of Mathematics \\
Columbia University \\
Rm 509, MC 4406 \\
2900 Broadway \\
New York, NY 10027 \\
U.S.A}
\thanks{The fourth author was partially supported by FCT doctoral grant SFRH/BD/68772/2010.}
\email{guerreiro@math.columbia.edu}

\author[D. Lowry-Duda]{David Lowry-Duda}
\address[David Lowry-Duda]{Department of Mathematics\\
Brown University\\
151 Thayer Street, Box 1917\\
Providence, RI 02912\\
U.S.A}
\thanks{The fifth author was supported by the National Science Foundation under Grant No. DGE 0228243.}
\email{djlowry@math.brown.edu}

\begin{document}

\begin{abstract}
    The zeta function of a curve $C$ over a finite field may be expressed in terms of the characteristic polynomial of a unitary matrix $\Theta_C$.
    We develop and present a new technique to compute the expected value of $\Tr(\Theta_C^n)$ for various moduli spaces of curves of genus $g$ over a fixed finite field in the limit as $g$ is large, generalizing and extending the work of Rudnick~\cite{rudnick} and Chinis~\cite{chinis}.
    This is achieved by using function field zeta functions, explicit formulae, and the densities of prime polynomials with prescribed ramification types at certain places as given in~\cite{BDFKLOW} and~\cite{Zhao}.
    We extend~\cite{BDFKLOW} by describing explicit dependence on the place and give an explicit proof of the Lindel\"{o}f bound for function field Dirichlet $L$-functions $L(1/2 + it, \chi)$.
    As applications, we compute the one-level density for hyperelliptic curves, cyclic $\ell$-covers, and cubic non-Galois covers.
\end{abstract}

\maketitle
\section{Introduction and statement of results}

Let $\mathbb{F}_q$ be a finite field of odd cardinality, and let $C$ be a smooth curve over $\mathbb{F}_q$.
The Weil conjectures tell us that the Hasse-Weil zeta function has the form
\begin{equation*}
  Z_C (u)
  :=
  \exp\left(\sum_{n = 1}^\infty \#C(\mathbb{F}_{q^n}) \frac{u^n}{n} \right) =  \frac{P_C(u)}{(1-u)(1-qu)},
\end{equation*}
where
\begin{equation*}
  P_C(u)
  :=
  \det \bigl(1 - u \operatorname{Frob}| H^{1} _{\text{et}} \bigl( C\bigl( \overline{\FF_q} \bigr), \mathbb{Q}_\ell \bigr) \bigr) \in \mathbb{Z}[u]
\end{equation*}
is the characteristic polynomial of the Frobenius automorphism, whose roots have absolute value $q^{-1/2}$ and are stable (as a multiset) under complex conjugation.
Furthermore, $P_C( u)$ corresponds to a unique conjugacy class of a unitary symplectic matrix $\Theta_C \in \USp(2g)$ such that the eigenvalues $e^{i \theta_j}$ correspond to the zeros $q^{-1/2} e^{i \theta_j}$ of $P_C(u)$.
This conjugacy class $\Theta_C$ is called \emph{Frobenius class} of $C$.

For many different families of curves $C$, Katz and Sarnak~\cite{katz-sarnak} showed that as $q \to \infty$, the Frobenius classes $\Theta_C$ become equidistributed in certain subgroups of unitary matrices, where the group depends on the monodromy group of the family of curves.
Stated more precisely, suppose $\mathcal{F}(g,q)$ is a natural family of curves of genus $g$ over $\mathbb{F}_q$ with symmetry type $\M(2g)$, equipped with the Haar measure.
The expected value of a function $F$ evaluated on the eigenangles of curves in $\mathcal{F}(g,q)$ is defined as
\begin{equation*}
    \left\langle F(\Theta_C) \right\rangle _{\mathcal{F}(q,g)}
    :=
    \frac{1}{\#\mathcal{F}(q,g)}\sum_{C \in \mathcal{F}(q,g)} F(C).
\end{equation*}
Katz and Sarnak predicted that
\begin{equation*}
    \lim_{q \to \infty} \left\langle F(\Theta_C) \right\rangle_{\mathcal{F}(q,g)}
    =
    \int_{\M(2g)} F(U) \dd U,
\end{equation*}
where the integral is taken with the respect to the Haar measure.
This means that many statistics of the eigenvalues can be computed, in the limit, as integrals over the corresponding unitary monodromy groups.

One particularly important and well-studied statistic is the {one}-level density, which concerns low-lying zeroes. The definition of the one-level density $W_f(U)$ of a  $N \times N$ unitary matrix $U$ and with test function $f$ in the function field setting is given by~\eqref{def-one-level} in Section~\ref{sec:hyperelliptic}.

The work of Katz and Sarnak concerns the $q$-limit.
Recently, there has been work exploring another type of limit, examined by fixing a constant finite field $\F_q$ and looking
at statistics of families of curves as their genus $g \rightarrow \infty$, such as the work of Kurlberg and Rudnick~\cite{KURU} who first investigated that type
of limit for the distribution of $\Tr\bigl(\Theta_C\bigr)$ for the family of hyperelliptic curves. The statistics are then given by a sum of $q+1$ independent and identically distributed random
variables, and not as distributions in groups of random matrices.
In a subsequent work,
Rudnick~\cite{rudnick} investigated the distribution of  $\Tr\bigl(\Theta_C ^n\bigr)$ for the same family of hyperelliptic curves.
Denote by $\mathcal{F}_{2g + 1}$ the family of hyperelliptic curves of genus $g$ given in affine form by
\begin{equation*}
    C: Y^2 = Q(X)
\end{equation*}
where $Q(X)$ is a square-free, monic polynomial of degree $2g+1$.
Rudnick showed that the $g$-limit statistics for
trace of high powers   $ \Tr \bigl( \Theta_C ^n \bigr)$ over the family $\mathcal{F}_{2g + 1}$
agrees (for $n$ in a certain range) with the corresponding statistics over $\USp(2g)$ given by:
\begin{equation}
    \label{eqn:usp2g}
    \int_{\USp(2g)} \Tr U^n \dd U
    =
    \begin{cases}
        2g & n = 0,\\
        - \eta_n & 1 < |n| < 2g,\\
        0 & |n| > 2g.
    \end{cases}
\end{equation}
where
\begin{equation} \label{etan}
    \eta_n
    =
    \begin{cases}
        1 & n \text{ even},\\
        0 & n \text{ odd}.
    \end{cases}
\end{equation}

More precisely:
\begin{theorem}{\cite{rudnick}}
    \label{thm:rudnick}
    For all $n > 0$
    \begin{equation*}
        \left\langle
            \Tr \Theta_C ^n
        \right\rangle_{ \mathcal{F}_{2g + 1} }
        =
        \eta_n q^{-\frac{n}{2}} \sum_{\deg v \divides \frac{n}{2}} \frac{\deg v}{q^{\deg v} + 1}
        + O\bigl( g q^{-g} \bigr)
        +
        \begin{cases}
            - \eta_n                        & 0 < n < 2g, \\
            -1 - \frac{1}{q - 1}            & n = 2g, \\
            O\bigl( n q^{\frac{n}{2} - 2g} \bigr) & n > 2g,
        \end{cases}
    \end{equation*}
    where the sum is over all finite places of $\mathbb{F}_q[X]$.

    Furthermore, if $3 \log_q g < n < 4 g - 5 \log_q g$ and $n \neq 2g$, then
    \begin{equation*}
        \left \langle \Tr \Theta_C ^n \right \rangle = \int_{\USp(2g)} \Tr U^n \dd U + o\left( \frac{1}{g} \right).
    \end{equation*}
    Moreover, if $f$ is an even test function in the Schwartz space $\mathcal{S}(\mathbb{R})$ with Fourier transform $\hat{f}$ supported in $(-2,2)$, then
    \begin{equation*}
      \left\langle W_f \right\rangle_{\mathcal{F}_{2g+1}} = \int_{\USp(2g)} W_f(U) \dd U + \frac{\operatorname{dev}(f)}{g} + o\left( \frac{1}{g} \right),
    \end{equation*}
    where
    \begin{equation*}
        \operatorname{dev}(f) = \hat{f}(0)\sum_v \frac{\deg v}{q^{2 \deg v} - 1} - \hat{f}(1)\frac{1}{q-1}
    \end{equation*}
    and the sum is over all finite places of $\mathbb{F}_q[X]$.
\end{theorem}

We remark that the bias towards having more points over $\F_{q^n}$ whenever $n$ is even (and in a certain range with respect to $g$), which follows from the symplectic symmetry, was first pointed out by Brock and Granville~\cite{bg}.

The results of Rudnick~\cite{rudnick} hold for statistics over the space  $\mathcal{F}_{2g + 1}$, which is only a subset of the moduli space of hyperelliptic curves of genus $g$,
$\mathcal{H}_g$ (cf. Section~\ref{sec:hyperelliptic}).
The statistics for the whole moduli space of hyperelliptic curves of genus $g$, $\mathcal{H}_g$, were obtained by Chinis
\cite{chinis}, and they differ slightly from the statistics for $\mathcal{F}_{2g + 1}$.
\begin{theorem}{\cite{chinis}}
    \label{thm:chinis}
    For $n$ odd,
    \begin{equation*}
        \left\langle \Tr\bigl(\Theta_C^n\bigr) \right\rangle _{\mathcal{H}_g} = 0,
    \end{equation*}
    and for $n$ even,
    \begin{equation*}
        \left\langle \Tr \bigl( \Theta_C^n \bigr) \right\rangle_{\mathcal{H}_g}
        =
        q^{-\frac{n}{2}}  \sum_{\substack{{\deg v \divides\frac{n}{2}}\\{\deg v \neq 1}}} \frac{\deg v}{q^{\deg v}+1}
        + O\bigl( g q^{\frac{-g}{2}} \bigr)
        +
        \begin{cases}
            -1                                  & 0 < n < 2g,\\
            -1-\frac{1}{q^2-1}                  & n = 2g,\\
            O\bigl( nq^{\frac{n}{2}-2g} \bigr) & 2g < n,
        \end{cases}
    \end{equation*}
    where the sum is over all finite places of $\mathbb{F}_q[X]$.
\end{theorem}

It is interesting that when studying the distribution of zeta zeros for hyperelliptic curves Faifman and Rudnick~\cite{fr} can restrict to half of the moduli space (in this case, polynomials of even degree) without it affecting the result; but when one restricts to $\mathcal F_{2g+1}$ the one-level density is not quite the same as the one-level density on the whole moduli space $\mathcal H_g$. The difference is explained by the fact that the infinite place behaves differently in $\mathcal F_{2g+1}$ and $\mathcal F_{2g+2}$.
work of
The results of Rudnick were vastly generalized in a recent paper of Bui and Florea \cite{Bui-Florea}, which give formulas for the one-level density which are uniform in $q$ and $d$, and they can then identify lower order terms when the support of the test function holds in various ranges. For the one-level density of classical Dirichlet L-functions associated to quadratic characters, some recent work of Fiorilli, Parks and Sodergren~\cite{FPS} exhibits all the lower order terms which are descending powers of $\log{X}$.

We are interested in a different generalization, extending the statistics of Rudnick and Chinis
to statistics of families of curves for fixed $q$ and as $g$ varies.
In Section~\ref{sec:hyperelliptic}, we first present a new proof for Theorem~\ref{thm:chinis} using function field zeta functions and explicit formulae, specifically relying on densities of prime polynomials of different ramification types, as described in~\cite{BDFKLOW}.
Our technique is much simpler than what is used in \cite{rudnick} and \cite{Bui-Florea}, and the result presented in Section~\ref{sec:hyperelliptic} is  weaker than the results of Rudnick and Chinis (as our result holds for a more limited range of $n$), but it has the benefit of having clear generalization to many families of curves.
We present two such generalizations here.
In Section~\ref{sec:general-ell} we generalize the result from Section~\ref{sec:hyperelliptic} for cyclic $\ell$-covers curve, and in Section~\ref{sec:cubic_nongalois}, we do the same for cubic curves corresponding to non-Galois extensions. We summarize our main results in the following theorems, whose details can be found in Section~\ref{sec:general-ell}
and~\ref{sec:cubic_nongalois}. Throughout this paper, all explicit constants in the error terms can depend on $\ell$ and $q$.

\begin{theorem}
    \label{introthm:Hell} Let $\ell$ be an odd prime and let ${\mathcal{H}_{g,\ell}}$ be the moduli space of $\ell$ covers of genus $g$.
    For any $\varepsilon > 0$ and $n$ such that $6 \log_q{g} < n < (1 - \varepsilon) \left(\frac{2 g}{\ell - 1} + 2 \right)$, as $g \rightarrow \infty$ we have
    \begin{align*}
        \left\langle \Tr \Theta_C ^n \right\rangle_{\mathcal{H}_{g,\ell}}
        &=  \int_{\U(2g)} \Tr U^n \dd U + O\left( \frac{1}{g} \right).
    \end{align*}
    Let $f$ be an even test function in the Schwartz space $\mathcal{S}(\mathbb{R})$ with $ \operatorname{supp} \hat{f} \subset \left(-\frac{1}{\ell - 1}, \frac{1}{\ell - 1} \right)$, then
    \begin{eqnarray*}
        \left \langle W_f (\Theta_C) \right\rangle_{\mathcal{H}_{g,\ell}} &=&  \int_{\U (2g)} W_f(U) \dd U \\
        &&- \hat{f}(0) \frac{\ell-1}{g} \sum_{v} \frac{\deg v}{(1 + (\ell-1) q^{-\deg v})(q^{\ell\deg v/2} - 1)}
        + O\left( \frac{1}{g^{2 - \varepsilon}} \right),
    \end{eqnarray*}
    where the sum is over all places $v$ of $\FF_q(X)$.
\end{theorem}

\begin{theorem}
  \label{introthm:E3} Let $E_3(g)$ be the space of cubic non-Galois extensions of $\F_q(X)$ with discriminant of degree $2g+4$,  and let $\delta,B > 0$ be fixed constants as in Theorem~\ref{thm:zhao}.

 For $6 \log_q{g} < n < \frac{\delta g}{B+1/2}$, and as $g \rightarrow \infty$,
    \begin{equation*}
        \left\langle \Tr \Theta_C ^n \right\rangle_{E_3(g)}
        =  \int_{\USp(2g)} \Tr U^n \dd U + O\left( \frac{1}{g} \right).
    \end{equation*}
   Let $f$ be an even test function in the Schwartz space $\mathcal{S}(\mathbb{R})$ with $ \operatorname{supp} \hat{f} \subset (- \frac{\delta}{2B + 1}, \frac{\delta}{2B + 1} )$, then for any $\varepsilon > 0$,
   \begin{align*}
      \left \langle W_f (\Theta_C) \right\rangle_{E_3(g)} =  \int_{\USp (2g)} W_f(U) \dd U - \frac{\hat{f}(0)}{g}  \kappa  + O\left( \frac{1}{g^{2 - \varepsilon}} \right),
  \end{align*}
  where $\kappa$ is defined by \eqref{def-kappa}.
\end{theorem}

We remark that the one-level densities exhibit the predicted symmetries: unitary for cyclic covers of order $\ell$ (for $\ell$ an odd prime), and symplectic for cubic non-Galois extensions.

The main theorems of Sections~\ref{sec:hyperelliptic} and~\ref{sec:general-ell} rely on results concerning the densities of prime polynomials with prescribed ramification types at particular places from~\cite{BDFKLOW}, but also require understanding of dependence on those places.
In Section~\ref{sec:lindelof}, we show how to make explicit this dependence, which involves proving the explicit Lindel\"{o}f bound for $L(s, \chi)$.

\section*{Acknowledgements}
The proof of the Lindel\"{o}f Hypothesis presented in Section~\ref{sec:lindelof} was suggested to us by Soundararajan, following his work with Chandee for a similar bound for the Riemann zeta function~\cite{chandee-sound}, and we are very grateful for his suggestion and help. We would also like to thank Daniel Fiorilli and Lior Bary-Soroker for helpful discussions,
and Zeev Rudnick for helpful discussions related to this work, and comments on previous versions of this paper.
We would like to thank the Arizona Winter School for creating opportunities for research and providing an excellent platform for starting the collaboration that lead to this paper.
We also wish to thank ICERM (Providence, RI) for its hospitality during Fall 2015 when this paper was finalized.

\section{Hyperelliptic covers}
\label{sec:hyperelliptic}

In this section we present a weaker version of Theorem~\ref{thm:chinis}, using a different technique, namely using the
results of~\cite{BDFKLOW} to count the function field extensions corresponding to the hyperelliptic curves in $\mathcal{H}_g$ with prescribed ramification/splitting
conditions.
Let $\mathcal{H}_g$ be the moduli space of hyperelliptic curves of genus $g$. Every such curve has an affine model
\begin{equation*}
    C : Y^2 = Q(X),
\end{equation*}
with $Q(X)$ is a square-free polynomial of degree $2g+1$ or $2g+2$.

\begin{theorem}
    \label{thm:density-hyperelliptic}
    Let $E(\mathbb{Z}/2\mathbb{Z}, g)$ be the set of quadratic extensions of genus $g$ of $\mathbb{F}_q[X]$,
    let $v_0$ be a place, and
    let $E( \mathbb{Z}/2\mathbb{Z}, g, v_0, \omega)$ be the subset of $E(\mathbb{Z}/2\mathbb{Z}, g)$  with prescribed behavior $\omega
    \in \left\{ \text{ramified}, \text{split}, \text{inert} \right\}$
    at the place $v_0$.
    Then for any $\varepsilon > 0$,
    \begin{align*}
      \frac{\# E(\mathbb{Z}/2\mathbb{Z}, g, v_0,\text{ramified})}{\# E(\mathbb{Z}/2\mathbb{Z}, g)} &= \frac{ q^{- \deg v_0}}{1+q^{- \deg v_0}}
      +
      O \bigl( q^{-2g} \bigr)\\
      \frac{\# E(\mathbb{Z}/2\mathbb{Z}, g, v_0,\text{split})}{\# E(\mathbb{Z}/2\mathbb{Z}, g)}
      &=
      \frac{\# E(\mathbb{Z}/2\mathbb{Z}, g, v_0,\text{inert})}{\# E(\mathbb{Z}/2\mathbb{Z}, g)}\\ &= \frac{ 1}{2 \bigl( 1 + q^{- \deg v_0} \bigr)}
      +
      O \bigl( q^{(\varepsilon - 1) (g+1) + \varepsilon \deg v_0 } \bigr).
    \end{align*}
\end{theorem}

\begin{proof}
    It is shown in~\cite{BDFKLOW} that
    \begin{align*}
      \# E(\mathbb{Z}/2\mathbb{Z}, g) &= 2 q^{2g+2} \bigl( 1 - q^{-2} \bigr) \\
        \# E(\mathbb{Z}/2 \mathbb{Z}, g, v_0, \text{ramified})
        &=  \frac{ q^{- \deg v_0}}{1+q^{- \deg v_0}} 2 q^{2g+2} \bigl(1 - q^{-2}\bigr) + O(1) \\
      \# E(\mathbb{Z}/2\mathbb{Z}, g, v_0, \text{split}) &= \# E(\mathbb{Z}/2\mathbb{Z}, g, v_0, \text{inert}) \\
      &= \frac{ 1}{ 2 \bigl( 1 + q^{- \deg v_0} \bigr) }  2 q^{2g+2} \bigl( 1 - q^{2} \bigr)
      + O_{v_0} \bigl( q^{(g + 1)(1 + \varepsilon)}  \bigr),
    \end{align*}
    where $O_{v_o}$ indicates that the implicit constant may depend on $v_0$.
    We prove in Section~\ref{sec:lindelof} that keeping track of the dependence on $v_0$ gives
    \begin{equation*}
      \# E(\mathbb{Z}/2\mathbb{Z}, g, v_0, \omega) = \frac{ 1}{2(1+q^{- \deg v_0})}
      2 q^{2g+2} (1 - q^{2}) + O \bigl( q^{(g+1)(1+\varepsilon) + \varepsilon \deg v_0} \bigr),
    \end{equation*}
    for $\omega \in \left\{ \text{split}, \text{inert} \right\}$ which proves the theorem.
\end{proof}
\begin{lemma}
    \label{lemma:hyperelliptic}
    Let $C$ be a fixed $\F_q$-point in the moduli space $\mathcal{H}_{g}$, $\FF_q(C)$ its function field and $\Tr \Theta_C ^n$ be the $n$-th power of the trace of $C$.
    Then
    \begin{equation}
        \label{eqn:expansion-hyperelliptic}
        -q^{\frac{n}{2}} \Tr \Theta_C ^n
        =
        \sum_{{\substack{\deg v \divides n\\ v \text{ split in }\\ \FF_q(C)  }}} \deg v
        +
        \eta_n \sum_{{\substack{\deg v \divides  \frac{n}{2} \\ v \text{ inert in }\\ \FF_q(C) }}} 2 \deg v
        -
        \sum_{{\substack{\deg v \divides  n \\ v \text{ inert in }\\ \FF_q(C) }}} \deg v,
    \end{equation}
    where the sums are over all places $v$ of $\FF_q(X)$ (including the place at infinity) with the prescribed behavior.
\end{lemma}
\begin{proof}
    For any function field $K$, over $\FF_q(X)$,  we  denote its zeta function by $\zeta_K(s)$.
    The lemma follows by taking the logarithmic derivative on both sides of
    \begin{equation*}
        \prod_{j=1} ^{2g} \bigl( 1 - q^{1/2} q^{-s} e^{i \theta_j} \bigr) = P_C\bigl( q^{-s} \bigr) = \frac{\zeta_{\FF_q(C)} (s)}{\zeta_{\FF_q(X)} (s)}
    \end{equation*}
    with respect to $q^{-s}$ after expressing $\zeta_{\FF_q(C)}(s)/\zeta_{\FF_q(X)}(s)$ as an Euler product.
\end{proof}

\begin{theorem}
    The average $n$-th moment of the trace over hyperelliptic curves of genus $g$ is given by
    \label{thm:moment-hyperelliptic}
    \begin{align*}
        \left\langle -q^{\frac{n}{2} } \Tr \Theta_C ^n \right\rangle_{\mathcal{H}_{g}}
        & =
        \eta_n q^{\frac{n}{2}} -  \eta_n \sum_{\substack{\deg v \divides  \frac{n}{2} \\\deg v \neq 1} } \frac{\deg v}{1+q^{\deg v}} +  O\bigl( q^{(\varepsilon-1)(g+1) + n(1 + \varepsilon)} \bigr)
    \end{align*}
    for all $\varepsilon > 0$, and where the sum is over all finite places $v$ of $\FF_q(X)$.
\end{theorem}

\begin{proof}
	We start out by averaging equation~\eqref{eqn:expansion-hyperelliptic} over hyperelliptic curves of genus $g$, hence $\left\langle -q^{\frac{n}{2} } \Tr \Theta_C ^n \right\rangle_{\mathcal{H}_{g}}$ equals
    \begin{equation*}
        \frac{1}{\# E(k, \mathbb{Z}/2\mathbb{Z}, g)}
        \sum_{C \in \mathcal{H}_g}
         \left(
         \sum_{\substack{\deg v \divides n\\ v \text{ split in }\\ \FF_q(C)  }} \deg v
        +
         \eta_n \sum_{\substack{\deg v \divides  \frac{n}{2} \\ v \text{ inert in }\\ \FF_q(C) }} 2 \deg v
         - \sum_{\substack{\deg v \divides  n \\ v \text{ inert in }\\ \FF_q(C) }} \deg v \right).
    \end{equation*}
    Swapping the order of summation gives us
    \begin{align*}
        \left\langle -q^{\frac{n}{2} } \Tr \Theta_C ^n \right\rangle_{\mathcal{H}_{g}}
        &
        = \sum_{\deg v \divides n} \deg v \frac{\# E(\mathbb{Z}/2\mathbb{Z}, g, v, \text{split})}{\# E(k, \mathbb{Z}/2\mathbb{Z}, g)}\\
        &
        \quad +  \eta_n \sum_{\deg v \divides \frac{n}{2}} 2\deg v \frac{\# E(\mathbb{Z}/2\mathbb{Z}, g, v, \text{inert})}{\# E(\mathbb{Z}/2\mathbb{Z}, g)}\\
        &
        \quad - \sum_{\deg v \divides n} \deg v \frac{\# E(\mathbb{Z}/2\mathbb{Z}, g, v, \text{inert})}{\# E(\mathbb{Z}/2\mathbb{Z}, g)}.
    \end{align*}
    Applying Theorem~\ref{thm:density-hyperelliptic} we get
    \begin{align*}
        \left\langle
            -q^{\frac{n}{2} } \Tr \Theta_C ^n
        \right\rangle_{\mathcal{H}_{g}}
        &=
        \sum_{\deg v \divides n} \deg v \left( \frac{1}{2(1+q^{-\deg v})} +  O\bigl( q^{(\varepsilon-1)(g+1) + \varepsilon \deg v}\bigr) \right)
        \\
        & \quad +
        \eta_n \sum_{\deg v \divides \frac{n}{2}} \deg v \left( \frac{1}{1+q^{-\deg v}}  +  O\bigl( q^{(\varepsilon-1)(g+1)  + \varepsilon \deg v }\bigr) \right)\\
        & \quad -
        \sum_{\deg v \divides n} \deg v \left( \frac{1}{2(1+q^{-\deg v})} +  O\bigl( q^{(\varepsilon-1)(g+1)  + \varepsilon \deg v}\bigr) \right).
    \end{align*}
    The main terms of the first and the third sums cancel, but their error terms do not.
    Therefore
    \begin{align*}
        \left\langle
            -q^{\frac{n}{2} } \Tr \Theta_C ^n
        \right\rangle_{\mathcal{H}_{g}}
        &=
        \eta_n \sum_{\deg v \divides \frac{n}{2}}  \frac{\deg v}{1+q^{-\deg v}}
         +
         O\left( q^{(\varepsilon-1)(g+1)} \sum_{\deg v \divides n} \deg v \: q^{\varepsilon \deg v}\right)\\
        & =
        \eta_n \sum_{\deg v \divides  \frac{n}{2}} \frac{\deg v}{1+q^{-\deg v}} +  O\bigl( q^{(\varepsilon-1)(g+1) + n( 1 + \varepsilon)} \bigr),
    \end{align*}
    where the last equality follows from the prime number theorem for $\mathbb{F}_q[X]$  (as proved in~\cite{rosen2002number} for instance).
    Using the following identity
    \begin{equation}
        \label{eqn:pnt}
        q^n = \sum_{d \divides  n} d \:  \pi(d),
    \end{equation}
    where $\pi(d)$ is the number of irreducible polynomials of degree $d$ defined over $\F_q$, we have for $n$ even that
    \begin{align*}
        \sum_{\deg v \divides  \frac{n}{2}} \frac{\deg v}{1+q^{-\deg v}}
        & = \sum_{\deg v \divides  \frac{n}{2}} \deg v - \sum_{\deg v \divides  \frac{n}{2}} \frac{\deg v }{1+q^{\deg v}}\\
        & =
        \sum_{d \divides  \frac{n}{2} } d \: \pi(d) + 1 - \sum_{\deg v \divides  \frac{n}{2}} \frac{\deg v }{1+q^{\deg v}}\\
        & =
        q^{\frac{n}{2}} -  \sum_{\substack{\deg v \divides  \frac{n}{2} \\\deg v \neq 1} } \frac{\deg v}{1+q^{\deg v}}.
    \end{align*}
 We remark that in the second equality above, the extra $1$ arises from the place at infinity.
\end{proof}

As expected from~\cite{rudnick} and~\cite{chinis}, the previous theorem agrees with corresponding statistics over $\USp(2g)$.
Recall from~\cite{eigenvalues} that
\begin{equation} \label{U-moments}
\int_{\USp(2g)} \Tr U^n \dd U = \begin{cases}
    2g & n = 0\\
    - \eta_n & 1 < n < 2g\\
    0 & n> 2g.
\end{cases}
\end{equation}

\begin{corollary}
    \label{cor:moment-hyperelliptic}
For any $\varepsilon > 0$, and as $g \rightarrow \infty$,
\begin{align*}
  \left\langle \Tr \Theta_C ^n \right\rangle_{\mathcal{H}_{g}}
  &=
  - \eta_n  \left( 1 - \frac{1}{1+q^{\frac{n}{2}}}\right) + O \left(q^{-\frac{n}{4}} + q^{(\varepsilon - 1) g + n (\varepsilon + \frac{1}{2})} \right).
  \end{align*}
  Moreover, for any $\varepsilon' >0$ and $n$ such that $4 \log_q{g} < n < 2g \bigl(1 - \varepsilon' \bigr)$, we have as $g \rightarrow \infty$
  \begin{align*}
  \left\langle \Tr \Theta_C^n \right\rangle_{\mathcal{H}_{g}}
    &=  \int_{\USp(2g)} \Tr U^n \dd U + O\left( \frac{1}{g} \right).
  \end{align*}
\end{corollary}

\begin{proof}
    Applying the prime number theorem to Theorem~\ref{thm:density-hyperelliptic}, we have
    \begin{equation*}
        \eta_n \sum_{\substack{\deg v \divides  \frac{n}{2} \\\deg v \neq 1} } \frac{\deg v}{1+q^{\deg v}} = \eta_n \frac{ q^{\frac{n}{2}} }{1+q^{\frac{n}{2}}} + \eta_n O\left(q^{\frac{n}{4}}\right).
    \end{equation*}
    To prove the second statement, we apply the first statement choosing $\varepsilon$ small enough such that
        $(\varepsilon - 1) + 2\bigl(1 - \varepsilon'\bigr) \left(\varepsilon + \frac{1}{2}\right) < 0$.
\end{proof}

We can apply the last result to determine the one-level density of hyperelliptic curves, as done in \cite{rudnick}, and we recall the definition of the one-level density in the function field setting with the relevant properties below for completeness. We will also apply this to other families of curves in the following sections.

Let $f$ be an even test function in the Schwartz space $\mathcal{S}(\mathbb{R})$, and for any integer $N \geq 1$, we define
\begin{equation*} \label{def-1}
    F(\theta) := \sum_{k \in {\mathbb Z}} f\left( N \left( \frac{\theta}{2 \pi} - k \right) \right),
\end{equation*}
which has period $2 \pi$ and is localized in an interval of size approximatively $1/N$ in
${\mathbb R}/2 \pi {\mathbb Z}$. Then, for a unitary matrix $N \times N$ matrix $U$ with eigenvalues $e^{i \theta_j}, j=1, \dots, N$, we define the one-level density
\begin{equation}
    \label{def-one-level}
    W_f(U) := \sum_{j=1}^N F(\theta_j),
\end{equation}
counting the number of angles $\theta_j$ in an interval of length  approximatively $1/N$ around 0 (weighted with the function $f$).
Using the Fourier expansion, we have that
\begin{equation*}
    W_f(U) = \int_{-\infty}^\infty f(x) dx + \frac1N \sum_{n \neq 0} \hat{f} \left( \frac{n}{N} \right) \Tr U^n .
\end{equation*}
Katz and Sarnak conjectured that for any fixed $q$, the expected value of $W_f(\Theta_c)$ over ${\mathcal{H}_{g}}$ will converge to $\int_{\USp(2g)} W_f(U) \dd U$ as $g \rightarrow \infty$ for any test function, and we show in the next theorem that this holds for test functions on a limited support (which is more restrictive than the support obtained in \cite[Corollary 3]{rudnick}).

\begin{theorem}
    \label{thm:1level-density-hyperelliptic}
    Let $f$ be an even test function in the Schwartz space $\mathcal{S}(\mathbb{R})$ with $ \operatorname{supp} \hat{f} \subset (- 1, 1)$. Then for any $\varepsilon > 0$,
    \begin{align*}
        \left \langle W_f (\Theta_C) \right\rangle_{\mathcal{H}_{g}} =  \int_{\USp (2g)} W_f(U) \dd U + \frac{\hat{f}(0)}{g}  \sum_{\deg{v} \neq 1}  \frac{\deg{v}}{q^{ 2 \deg v} -1}   +
        O\left( \frac{1}{g^{2-\varepsilon}} \right),
    \end{align*}
    where the sum is over all finite places $v$ of $\FF_q(X)$.
    Moreover,
    \begin{align*}
        \lim_{g \rightarrow \infty} \left \langle W_f (\Theta_C) \right\rangle_{\mathcal{H}_{g}}
        & = \lim_{g \rightarrow \infty} \int_{\USp(2g)} W_f(U) \dd U\\
        & = \int_{\mathbb{R}} f(x) \left( 1 - \frac{\sin (2 \pi x)}{2 \pi x} \right) \dd x.
    \end{align*}
\end{theorem}

As we mentioned in the introduction, a vast generalization of the formula above was obtained by Bui and Florea \cite{Bui-Florea} in some recent work.

\begin{proof}
    As $\hat{f}$ is continuous, its support is contained in $[-\alpha, \alpha]$ for some $0 < \alpha < 1$. Using the Fourier expansion as above, we get
    \begin{equation}
    \label{formula-for-OLD}
    \begin{aligned}
        W_f(\Theta_C)
        &
        = \sum_{j=1} ^{2g} \sum_{k \in \mathbb{Z}} f\left(2g\left(\frac{\theta_j}{2 \pi} - k \right)\right)
        \\
        &
        = \int_\mathbb{R} f(x) \dd x + \frac{1}{2g} \sum_{n \neq 0} \hat{f}\left( \frac{n}{2g} \right) \Tr \Theta^n _C
        \\
        &
        = \hat{f}(0) + \frac{1}{g} \sum_{n=1} ^{2\alpha g}  \hat{f} \left( \frac{n}{2g} \right) \Tr \Theta_C ^n ,
    \end{aligned}
\end{equation}
    where the last equality follows from $f$ being even and the condition on the support of $\hat{f}$.
    Averaging $W_f(\Theta_C)$ over our family of curves and applying Theorem \ref{thm:moment-hyperelliptic} with
    \begin{equation*}
        0 < \varepsilon < \frac{1-\alpha}{2+2\alpha},
    \end{equation*}
    we get
    \begin{align*}
        \left\langle W_f(\Theta_C) \right\rangle_{\mathcal{H}_{g}} =& \hat{f}(0) - \frac{1}{g} \sum_{n=1} ^{\alpha g} \hat{f}\left( \frac{n}{g} \right)
        \\
        &+ \frac{1}{g}  \sum_{n=1} ^{\alpha g} \hat{f}\left( \frac{n}{g} \right) \frac{1}{q^n}  \sum_{\substack{\deg v \divides  {n} \\\deg v \neq 1} } \frac{\deg v}{1+q^{\deg v}}
        +O\left(q^{-\varepsilon g} \right)\\
        =&  \int_{\USp (2g)} W_f(U) \dd U
        \\
        &+ \frac{1}{g}  \sum_{n=1} ^{\alpha g} \hat{f}\left( \frac{n}{g} \right)  \frac{1}{q^n} \sum_{\substack{\deg v \divides  {n} \\\deg v \neq 1} } \frac{\deg v}{1+q^{\deg v}}
        +O\left(q^{-\varepsilon g} \right),
    \end{align*}
    where we note that by~\eqref{U-moments} and recalling that $f$ is even and $\operatorname{supp}\hat{f} \subset (-1, 1)$,
    \begin{equation*}
        \int_{\USp (2g)} W_f(U) \dd U = \hat{f}(0) - \frac{1}{g} \sum_{n = 1} ^{\alpha g}  \hat{f} \left( \frac{n}{g} \right).
    \end{equation*}

    We now compute
    \begin{multline} \label{first-step}
        \sum_{n=1}^{\alpha g} \hat{f} \left( \frac{n}{g} \right) \frac{1}{q^n} \sum_{\substack{\deg v \divides  {n} \\\deg v \neq 1} }  \frac{\deg{v}}{1 + q^{\deg v}}
        =\\
        \sum_{\substack {\deg{v} \leq \alpha g\\ \deg{v} \neq 1}}  \frac{\deg{v}}{1 + q^{\deg v}} \sum_{k \deg{v} \leq \alpha g}\hat{f} \left( \frac{k \deg{v}}{g} \right) \frac{1}{q^{k \deg v}}
    \end{multline}
    Suppose $\phi(g)$ is a function tending to $0$ as $g$ tends to infinity, to be specified later.
    We break the range of the inside sum of the right hand side at $g \phi(g)$.
    For the first range, we use the Taylor expansion for $\hat f$ to write $\hat{f} (x) = \hat{f}(0) + O(x) = \hat{f}(0) + o(1)$, explicitly,
    \begin{equation*}
        \hat{f} \left( \frac{k \deg{v}}{g} \right) = \hat{f}(0) + O \left( \frac{k \deg{v}}{g} \right).
    \end{equation*}
    Thus, \eqref{first-step} can be rewritten as
    \begin{eqnarray*}
       && \left( \hat{f}(0) + O \left( \phi(g) \right) \right)
        \sum_{\substack {\deg{v} \leq \alpha g\\ \deg{v} \neq 1}}
        \frac{ \deg{v} }{ 1 + q^{\deg v} }
        \left(
        \frac{1}{q^{\deg v} - 1}
        +
        O \left( q^{-g \phi(g)} \right)
        \right)\\
    &&=\hat{f}(0)
        \sum_{\substack {\deg{v} \leq \alpha g\\ \deg{v} \neq 1}}
        \frac{\deg{v}}{q^{2 \deg v} - 1}
        + O\left(\phi(g) + q^{-g \phi(g)}\right)\\
   &&=
        \hat{f}(0)
        \sum_{{\deg{v} \neq 1}}
        \frac{ \deg{v} }{ q^{2 \deg v} - 1}
        +
        O\left(\phi(g) + q^{-g \phi(g)} + q^{- 2 \alpha g}\right).
    \end{eqnarray*}
    For the remaining range,
    \begin{eqnarray*}
        &&\sum_{\substack {\deg{v} \leq \alpha g\\ \deg{v} \neq 1}}
        \frac{\deg{v}}{q^{\deg v} + 1}
        \sum_{g \phi(g) \leq k \deg{v} \leq \alpha g}
        \hat{f}
        \left( \frac{k \deg{v}}{g} \right) \frac{1}{q^{k \deg v}} \\ &&\ll
        \sum_{\substack {\deg{v} \leq \alpha g\\ \deg{v} \neq 1}}  \frac{\deg{v}}{q^{\deg v} + 1}
        \sum_{g \phi(g) \leq k \deg{v} \leq \alpha g}
        q^{-g \phi(g)}
        \\
        &&\ll  \alpha g q^{-g \phi(g)}.
    \end{eqnarray*}
    Thus, by choosing $\phi(g) = g^{-1 + \varepsilon}$, we get that
    \begin{equation*}
        \frac{1}{g} \sum_{n=1}^{\alpha g} \hat{f} \left( \frac{n}{g} \right) \frac{1}{q^n} \sum_{\substack{\deg v \divides  {n} \\\deg v \neq 1} }  \frac{\deg{v}}{ 1 + q^{\deg v} }
        =
    \frac{\hat{f}(0)}{g}   \sum_{\substack {\deg{v} \neq 1}}  \frac{\deg{v}}{q^{2 \deg v} - 1} + O \left( \frac{1}{g^{2-\varepsilon}} \right),
    \end{equation*}
    which proves the first statement.
    Taking the limit $g \rightarrow \infty$ we get the second part of the theorem.
\end{proof}

\section{General cyclic $\ell$-covers}
\label{sec:general-ell}

Let $\ell$ be an odd prime and assume that $q \equiv 1 \mod \ell$.
Let $\mathcal{H}_{g, \ell}$ be the moduli space of general $\ell$-covers of genus $g$.
Every such cover has an affine model
\begin{equation*}
  C : Y^\ell = Q(X),
\end{equation*}
where $Q( X )$ is an $\ell$-powerfree polynomial in $\FF_q[ X ]$.

We first state an explicit form of  Corollary 1.2 of~\cite{BDFKLOW} for the number of cyclic extensions with prescribed behavior at a given place $v_0$, keeping the dependence on the place $v_0$. All implied constants in the error term of this section can depend on $q$ and $\ell$.

\begin{theorem}
\label{thm:general-ell}
 Let $E(\mathbb{Z}/ \ell\mathbb{Z}, d)$ be the set of cyclic extensions of degree $\ell$ of $\mathbb{F}_q[X]$ with conductor of degree $d$,
 let $v_0$ be a place, $\omega \in  \{\text{ramified, split, inert}\}$,
  and
 $E( \mathbb{Z}/ \ell \mathbb{Z}, d, v_0, \omega)$ be the subset of $E( \mathbb{Z}/ \ell \mathbb{Z}, d)$  with prescribed behavior $\omega$ at the place $v_0$.
 Then for any $\varepsilon > 0$, we have
 \begin{equation*}
   \frac{\# E(\mathbb{Z}/ \ell \mathbb{Z}, d, v_0, \omega)}{\# E( \mathbb{Z}/ \ell \mathbb{Z}, d)} = c_{v_0, \omega} \frac{P_{v_0, \omega(d)}}{P(d)} + O\left( q^{ \left(\varepsilon - \frac{1}{2} \right) d + \varepsilon \deg v_0} \right),
 \end{equation*}
where
 \begin{align*}
   c_{v_0, \omega} = \begin{cases}
    \displaystyle \frac{ (\ell - 1) q^{- \deg v_0}}{1 + (\ell -1) q^{- \deg v_0}} & \text{if }\omega = \text{ ramified},\\
         \\
\displaystyle \frac{ 1}{ \ell (1 + (\ell - 1) q^{- \deg v_0})} & \text{if }\omega = \text{split or inert},
   \end{cases}
 \end{align*}
and where $P(x),  P_{v_0, \text{split}}(x),  P_{v_0, \text{ramified}}(x) \in \mathbb{R}[x]$ are monic polynomials of degree $\ell-2$ and
\begin{equation} \label{relation-SI}
P_{v_0,\text{inert}} (x) = (\ell - 1) P_{v_0, \text{split}}(x).\end{equation} Furthermore,
 \begin{equation} \label{bound-pp}
 \frac{P_{v_0, \text{inert}} (d)}{P(d)} = (\ell-1) + O
 \left( \frac{\deg{v_0}}{d} + \dots + \left( \frac{\deg{v_0}}{d} \right)^{\ell-2} \right).\end{equation}

Finally, if $\omega = \text{ramified}$, the error term can be written as  $O\left( q^{ \left(\varepsilon - \frac{1}{2} \right) d} \right)$, i.e., there is not dependence on the place $v_0$ in that  case.
\end{theorem}

\begin{proof}
This follows from Corollary 1.2 of~\cite{BDFKLOW}, keeping the dependence of the error term on the place $v_0$ as done in Section~\ref{sec:lindelof}. This gives
\begin{align*}
    \# E( \mathbb{Z}/ \ell \mathbb{Z}, d) &=  C_\ell q^{d} P(d) + O \left( q^{\left(\frac{1}{2} + \varepsilon \right)d} \right)\\
    \# E( \mathbb{Z}/ \ell \mathbb{Z}, d, v_0, \text{ramified}) &=  c_{v_0, \omega} C_\ell q^{d} P_{v_0,\omega} (d) + O \left( q^{\left(\frac{1}{2} + \varepsilon \right)d} \right)\\
    \# E( \mathbb{Z}/ \ell \mathbb{Z}, d, v_0, \text{split}) &=
  \# E( \mathbb{Z}/ \ell \mathbb{Z}, d, v_0, \text{inert})\\
  &= c_{v_0, \omega} C_\ell q^d P_{v_0,\omega}(d) + O \left( q^{\left( \frac{1}{2} + \varepsilon \right) d + \varepsilon \deg v_0} \right)
\end{align*}
To bound the quotient $\displaystyle \frac{P_{v_0, \text{inert}} (d)}{P(d)}$,
we also need the dependence on the coefficients of
 $$P_{v_0, \text{inert}} (x) = (\ell-1) x^{\ell-2} + a_{v_0, \ell-3} x^{\ell-3} + \dots + a_{v_0, 0}$$ for the place $v_0$. It follows from the computations of~\cite{BDFKLOW} on page 4327 that $$a_{v_0,i} \ll \left( \deg{v_0} \right)^{\ell-2-i} \;\;\mbox{for $0 \leq i \leq \ell-3$}.$$  (This comes from the residue computation at $u=q^{-1}$). The bound \eqref{bound-pp} then follows.
\end{proof}

Recall that for a function field extension $L/K$ cyclic of order $\ell$, the discriminant and conductor of $L/K$ are related by
\begin{equation*}
    \deg \operatorname{Disc}(L/K) = (\ell-1) \deg \operatorname{Cond}(L/K),
\end{equation*}
(as given in Theorem~7.16 of~\cite{rosen2002number}) and from the Riemann--Hurwitz formula, we have
\begin{equation*}
    2g + 2 (\ell-1) = \deg \operatorname{Disc}(L/K).
\end{equation*}
Thus we can interpret Theorem~\ref{thm:general-ell} in terms of the genus $g$ by taking
\begin{equation}
    \label{relation-dg}
    d = \frac{2g}{\ell-1} + 2.
\end{equation}
Further, put
\begin{equation*}
    \lambda_n :=
    \begin{cases}
        1 & \ell \divides  n,\\
        0 & \text{otherwise}.
    \end{cases}
\end{equation*}

\begin{lemma}
    \label{lemma:traces-ell}
    Let $C$ be a given curve in $\mathcal{H}_{g,\ell}$, $\FF_q(C)$ its function field and $\Tr \Theta_C ^n$ be the $n$-th power of the trace of $C$.
    Then
  \begin{equation}
    \label{eqn:expansion-general}
    -q^{\frac{n}{2}} \Tr \Theta_C ^n
    =
    (\ell-1) \sum_{\substack{ \deg v \divides n\\ v \text{ split in } \\ \FF_q(C) }} \deg v
    + \lambda_n  \ell \sum_{\substack{\deg v \divides  \frac{n}{\ell} \\ v \text{ inert in } \\ \FF_q(C) }} \deg v
    - \sum_{\substack{\deg v \divides  n \\ v \text{ inert in } \\ \FF_q(C) }} \deg v,
  \end{equation}
  where the sums are over all places $v$ of $\FF_q(X)$ (including infinity) with the prescribed behavior.
\end{lemma}
\begin{proof}
    \emph{Mutatis mutandis} Lemma~\ref{lemma:hyperelliptic}.
\end{proof}

\begin{theorem} For any $\epsilon > 0$, we have
    \label{theorem:moment-ell}
    \begin{align*}
        \left\langle -q^{\frac{n}{2}} \Tr \Theta _C ^{n} \right\rangle_{\mathcal{H}_{g,\ell}} &=
        \lambda_n
        \sum_{\deg v \divides  \frac{n}{\ell}}
        \frac{(\ell-1) \deg v}{1+(\ell-1)q^{-\deg v}} +
         O \left( \frac{q^{n/\ell} n^{\ell-2}}{d} +
        q^{\left( \varepsilon - \frac{1}{2} \right)d + n(1 + \varepsilon)}
        \right), \end{align*}
    where $d$ is defined by~\eqref{relation-dg}.
\end{theorem}

\begin{proof}
    We average \eqref{eqn:expansion-general} over $E(\mathbb{Z} / \ell \mathbb{Z}, d)$
    with $d=2g/(\ell-1) + 2$ using Theorem~\ref{thm:general-ell} to obtain
    \begin{align*}
        \left\langle -q^{\frac{n}{2} } \Tr \Theta_C ^n \right\rangle_{\mathcal{H}_{g,\ell}}
        = &
        (\ell-1) \sum_{\deg v \divides n}
        \deg v
        \left(
        c_{v, \text{split}}
        \frac{P_{v,\text{split}}(d)}{P(d)}
        +
        O \left(
        q^{\left( \varepsilon - \frac{1}{2} \right)d + \varepsilon \deg v}
        \right)
        \right)\\
        &+
        \lambda_n \ell\sum_{\deg v \divides  \frac{n}{\ell}}
        \deg v
        \left(
        c_{v, \text{inert}}
        \frac{P_{v, \text{inert}}(d)}{P(d)} +O \left(q^{\left( \varepsilon - \frac{1}{2} \right)d + \varepsilon \deg v} \right)
        \right)\\
        &-
        \sum_{\deg v \divides  n}
        \deg v
        \left(
        c_{v, \text{inert}}
        \frac{P_{v,\text{inert}}(d)}{P(d)}  + O \left( q^{\left( \varepsilon - \frac{1}{2} \right)d + \varepsilon \deg v} \right)
        \right).
    \end{align*}
    Since $c_{v, \text{inert}} = c_{v, \text{split}}$ and $P_{v,\text{inert}}(x) = (\ell - 1) P_{v,\text{split}}(x)$, the main term in the first and the third sum cancel. Thus, using \eqref{bound-pp}
    \begin{align*}
        \left\langle -q^{\frac{n}{2} } \Tr \Theta_C ^n \right\rangle_{\mathcal{H}_{g,\ell}}
        & = \lambda_n \ell
        \sum_{\deg v \divides  \frac{n}{\ell}} \frac{P_{v, \text{inert}}(d)}{P(d)}
        c_{v, \text{inert}}
        \deg v+
        O \left(
        q^{\left( \varepsilon - \frac{1}{2} \right)d}
        \sum_{\deg v \divides  n}
        \deg v
        q^{\varepsilon \deg v}
        \right)\\
        & = \lambda_n
        \sum_{\deg v \divides  \frac{n}{\ell}}
        \frac{(\ell-1) \deg v}{1+(\ell-1)q^{-\deg v}} + O \left( \frac{1}{d} \sum_{\deg v \divides  \frac{n}{\ell}}  \deg{v}^{\ell-2} \right)
        + O \left(
        q^{\left( \varepsilon - \frac{1}{2} \right)d + n(1 + \varepsilon)}
        \right). \\
        &= \lambda_n
        \sum_{\deg v \divides  \frac{n}{\ell}}
        \frac{(\ell-1) \deg v}{1+(\ell-1)q^{-\deg v}} +
         O \left( \frac{q^{n/\ell} n^{\ell-2}}{d} +
        q^{\left( \varepsilon - \frac{1}{2} \right)d + n(1 + \varepsilon)}
        \right). \\
    \end{align*}
\end{proof}

The previous theorem agrees with the corresponding statistics over the unitary group $\U(2g)$, as we have, by~\cite{eigenvalues},
\begin{equation*}
\int_{\U(2g)} \Tr U^n \dd U =
    \begin{cases}
        2g & n = 0,\\
        0 & n \neq 0.
    \end{cases}
\end{equation*}

\begin{corollary}
    \label{cor:moment-ellcovers}
    For any $\varepsilon > 0$ and $n$ such that $6 \log_q{g} < n < (1 - \varepsilon) \left(\frac{2 g}{\ell - 1} + 2 \right)$, as $g \rightarrow \infty$ we have
    \begin{align*}
        \left\langle \Tr \Theta_C ^n \right\rangle_{\mathcal{H}_{g,\ell}}
        &=  \int_{\U(2g)} \Tr U^n \dd U + O\left( \frac{1}{g} \right).
    \end{align*}
\end{corollary}

\begin{proof}
Using Theorem~\ref{theorem:moment-ell}, we have that
\begin{align*} \left\langle \Tr \Theta_C ^n \right\rangle_{\mathcal{H}_{g,\ell}} &=
        O \left( q^{n/\ell-n/2} n^{\ell-2} + q^{\left(2\varepsilon - 1 \right)g/(\ell-1) + n(1/2 + \varepsilon)} \right), \end{align*}
and we proceed as in the proof of Corollary~\ref{cor:moment-hyperelliptic}.
\end{proof}
\begin{theorem}
    \label{thm:1level-density-ell}
    Let $f$ be an even test function in the Schwartz space $\mathcal{S}(\mathbb{R})$ with $ \operatorname{supp} \hat{f} \subset \left(-\frac{1}{\ell - 1}, \frac{1}{\ell - 1} \right)$, then
    \begin{eqnarray*}
        \left \langle W_f (\Theta_C) \right\rangle_{\mathcal{H}_{g,\ell}} &=&  \int_{\U (2g)} W_f(U) \dd U \\
        &&- \hat{f}(0) \frac{\ell-1}{g} \sum_{v} \frac{\deg v}{(1 + (\ell-1) q^{-\deg v})(q^{\ell\deg v/2} - 1)}
        + O\left( \frac{1}{g^{2 - \varepsilon}} \right),
    \end{eqnarray*}
    where the sum is over all places $v$ of $\FF_q(X)$.
    Moreover,
    \begin{align*}
        \lim_{g \rightarrow \infty} \left \langle W_f (\Theta_C) \right\rangle_{\mathcal{H}_{g,\ell}}
        & = \lim_{g \rightarrow \infty} \int_{\U (2g)} W_f(U) \dd U\\
        & = \int_{\mathbb{R}} f(x) \dd x = \hat{f}(0).
    \end{align*}
\end{theorem}

\begin{proof}
    Pick $\alpha \in \left(0, \frac{1}{\ell - 1}\right)$, such that the support of $\hat{f}$ is contained in $[-\alpha, \alpha]$.
    By writing out the definition of the {one}-level density and obtaining the Fourier expansion for each variable $\theta_j$, we get
    \begin{align*}
        W_f(\Theta_C) &= \sum_{j=1} ^{2g} \sum_{k \in \mathbb{Z}} f\left(2g\left(\frac{\theta_j}{2 \pi} - k \right)\right) \\
        &= \int_\mathbb{R} f(x) \dd x + \frac{1}{2g} \sum_{n \neq 0} \hat{f}\left( \frac{n}{2g} \right) \Tr \Theta_C ^n \\
        &= \hat{f}(0) + \frac{1}{g} \sum_{n=1}^{2\alpha g} \hat{f} \left( \frac{n}{2g} \right) \Tr \Theta_C ^n,
    \end{align*}
    where the last equality follows from $f$ being even and the condition on the support of $\hat{f}$.

    Averaging $W_f(\Theta_C)$ over our family of curves and applying Theorem~\ref{theorem:moment-ell} with
    \begin{equation*}
        0 < \varepsilon < \frac{1 - \alpha(\ell - 1) }{\ell+1+2\alpha(\ell - 1)},
    \end{equation*}
    we get
    \begin{equation*}
        \begin{aligned}
            \left\langle W_f(\Theta_C) \right\rangle_{\mathcal{H}_{g,\ell}}
            = &
            \hat{f}(0)
            -
            \frac{\ell-1}{g} \sum_{n = 1}^{2\alpha g/\ell} \hat{f}\bigg(\frac{\ell n}{2g}\bigg) \frac{1}{q^{\ell n/2}}  \sum_{\deg v \mid n} \frac{\deg v} {1 + (\ell-1)q^{-\deg v}}
            \\
            & + O \left( \frac{1}{g^2} \sum_{n=1}^{2 \alpha g} \hat{f}\bigg(\frac{n}{2g}\bigg)
            q^{n/\ell-n/2} n^{\ell-2} \right)
             +
            O(q^{-\varepsilon g}) \\
            = &
            \hat{f}(0)
            -
            \frac{\ell-1}{g} \sum_{n = 1}^{2\alpha g/\ell} \hat{f}\bigg(\frac{\ell n}{2g}\bigg) \frac{1}{q^{\ell n/2}}  \sum_{\deg v \mid n} \frac{\deg v} {1 + (\ell-1)q^{-\deg v}}
            + O \left( \frac{1}{g^2} \right).
        \end{aligned}
    \end{equation*}
    We now compute
    \begin{multline*}
        \sum_{n = 1}^{2\alpha g/\ell} \hat{f}\bigg(\frac{\ell n}{2g}\bigg) \frac{1}{q^{\ell n/2}}  \sum_{\deg v \mid n} \frac{\deg v}{1 + (\ell-1)q^{-\deg v}} \\
    =
    \sum_{\deg v \leq 2\alpha g / \ell} \frac{\deg v}{1 + (\ell-1)q^{-\deg v}} \sum_{k \deg v \leq 2\alpha g / \ell} \hat{f}\left(\frac{k \ell \deg v}{2g}\right) \frac{1}{q^{\ell k \deg v/2}}.
    \end{multline*}
    As in the proof of Theorem~\ref{thm:1level-density-hyperelliptic}, let $\phi(g)$ be a function which tends to $0$ as $g$ tends to $\infty$, and we split the range of the inner sum at $g \phi(g)$.
    We start by addressing the first range, $k\deg v \leq g \phi(g)$.
    From the Taylor expansion of $\hat{f}(x)$ at 0, we have
    \begin{equation*}
        \hat{f}\left(\frac{k \ell \deg v}{2g}\right) = \hat{f}(0) + O\left( \frac{k \deg v}{g} \right),
    \end{equation*}
    thus
    \begin{eqnarray*}
        &&\sum_{\deg v \leq 2\alpha g / \ell} \frac{\deg v}{1 + (\ell-1)q^{-\deg v}} \sum_{k\deg v \leq g \phi(g)} \hat{f}\left(\frac{k \ell \deg v}{2g}\right) \frac{1}{q^{\ell k \deg v/2}}\\
    &&=
        \left(\hat{f}(0) + O(\phi(g))\right) \sum_{\deg v \leq 2 \alpha g/\ell} \frac{\deg v}
        {1 + (\ell-1)q^{-\deg v}} \left(\frac{1}{q^{\ell \deg v/2} - 1} + O\left(q^{-\ell g \phi(g)/2}\right)\right) \\
     && =  \hat{f}(0) \sum_{\deg v \leq 2\alpha g/\ell}
     \frac{\deg v}{(1 + (\ell-1)q^{-\deg v})(q^{\ell\deg v/2} - 1)} + O\left(\phi(g) + q^{-g\phi(g)}\right) \\
     && = \hat{f}(0) \sum_{v} \frac{\deg v}{(1 + (\ell-1)q^{-\deg v})(q^{\ell\deg v/2} - 1)} + O\left(\phi(g) + q^{-g\phi(g)} + q^{-\frac{(2 + \ell)\alpha g}{\ell}}\right).
    \end{eqnarray*}
    For the remaining range,
    \begin{eqnarray*}
        &&\sum_{\deg v \leq 2\alpha g/\ell} \frac{\deg v}{1 + (\ell-1)q^{-\deg v}}\sum_{g\phi(g) \leq k \deg v \leq 2\alpha g / \ell} \hat{f}\left(\frac{k \ell \deg v}{2 g}\right) \frac{1}{q^{\ell k \deg v / 2}}\\
        && \ll \sum_{\deg v \leq 2 \alpha g / \ell} \frac{\deg v}{1 + (\ell-1)q^{-\deg v}} \sum_{g\phi(g)\leq k \deg v \leq 2 \alpha g / \ell} q^{-k \deg v} \\
        && \ll \alpha g q^{-g\phi(g)}.
    \end{eqnarray*}
    Using $\phi(g) = g^{-1 + \epsilon}$, this completes the proof of the first statement of the theorem.
    As
    \begin{equation*}
    \lim_{g \to \infty} \int_{\U(2g)}W_f(U) \dd U = \int_\mathbb{R} f(x) \dd x = \hat{f}(0),
    \end{equation*}
    we get the second part of the theorem by taking the limit $g \to \infty$.
\end{proof}

\section{Cubic non-Galois Covers}
\label{sec:cubic_nongalois}

In this section, we consider the family of cubic non-Galois curves.
As a first step we need to count the number of cubic non-Galois extensions of genus $g$ of $\F_q(X)$ with prescribed splitting at given places $v$ with an explicit error term (in the genus $g$ and in the place $v$).
The following result was recently obtained by Zhao~\cite{Zhao}.
The count was previously established by Datskovsky and Wright~\cite{datskovsky1988density}, but without an error term which is needed for the present application.
As the final version of the preprint~\cite{Zhao} is not available, we write the explicit constants appearing in the error term as general constants, $\delta$ for the power saving in the count, and $B$ for the dependence on the place $v$.
This allows to get a general result that could be applied to different versions of Theorem~\ref{thm:zhao}. The same convention was adopted by Yang~\cite{yang2009distribution} who considered the one level-density for cubic non-Galois extensions of $\QQ$, and this also allows us to compare our results with his.

\begin{theorem}
    \label{thm:zhao}{\cite{Zhao}}
    Let $E_3(g)$ be the set of cubic non-Galois extensions of $\FF_q(X)$ with discriminant of degree $2g + 4$.
    For any finite set of primes $\mathcal{S}$, and any set $\Omega$ of splitting conditions for the primes contained in $\mathcal{S}$, define $E_3(g,\mathcal{S}, \Omega)$ to be the subset of $E_3(g)$ consisting of the cubic extensions satisfying those splitting conditions. Then, as $g \rightarrow \infty,$
    \begin{equation*}
        \frac{\# E_3(g,\mathcal{S}, \Omega)}{\# E_3(g)} = \prod_{v \in \mathcal{S}} c_v + O\left(q^{-\delta g} \prod_{v \in \mathcal{S}} q^{B \deg v} \right),
    \end{equation*}
    where $\delta, B > 0$ are fixed constants, and
    \begin{equation*}
        c_v = \frac{q^{2 \deg v}}{1 + q^{\deg v} + q^{2 \deg v}}
        \begin{cases}
            1/6 & v\text{ totally split},\\
            1/2 & v\text{ partially split},\\
            1/3 & v\text{ inert},\\
            q^{- \deg v} & v\text{ partially ramified},\\
            q^{- 2 \deg v} & v\text{ totally ramified}.\\
        \end{cases}
    \end{equation*}
\end{theorem}

We also need the explicit formulas for the curves $C$ associated to the cubic non-Galois extensions in $E_3(g)$.
This is proven following exactly the same lines as the proofs of the explicit formulas for the families of elliptic curves and cyclic covers of order $\ell$ in lemmas \ref{lemma:hyperelliptic} and  \ref{lemma:traces-ell}.
The result can also be found in a paper of Thorne and Xiong \cite[Proposition 3]{thorne-xiong} who computed other statistics for the same family.
\begin{proposition}
    Let $C$ be a given curve with function field $\FF_q(C) \in E_3(g)$, and $\Tr \Theta_C ^n$ be the $n$-th power of the trace of $C$.
    Then
    \begin{equation}
        \label{eqn:thorne-xiong}
        \begin{split}
            -q^{\frac{n}{2}} \Tr \Theta_C ^n =&
            \sum_{\substack{\deg v \divides  n\\v \text{ \rm{totally split in} } \\ \FF_q(C) }} 2 \deg v
            + \sum_{\substack{\deg v \divides  \frac{n}{2}\\v \text{ \rm{partially split in} } \\ \FF_q(C) }} 2 \deg v
            \\
            &
            +
            \sum_{\substack{\deg v \divides  n\\ v \text{ \rm{partially ramified in} } \\ \FF_q(C) }}  \deg v
            +
            \sum_{\substack{\deg v \divides  \frac{n}{3}\\ v \text{ \rm{inert in} } \\ \FF_q(C) } }  3 \deg v
            -
            \sum_{\substack{\deg v \divides  n\\ v \text{ \rm{inert in} } \\ \FF_q(C) }} \deg v,
        \end{split}
    \end{equation}
    where the sums are over all places $v$ of $\FF_q(X)$ (including the place at infinity) with the prescribed behavior.
\end{proposition}

For convenience, write
\begin{equation*}
    \tau_n := \begin{cases}
        1 & 3 \divides n,\\
        0 & \text{otherwise},
    \end{cases}
\end{equation*}
and we recall that $\eta_n$ is given by \eqref{etan}.

\begin{theorem}\label{thm:trace-cubicnongalois} Let $\delta, B >0$ be as in Theorem \ref{thm:zhao}.
  The average $n$-th moment of the trace over cubic non-Galois curves in $E_3(g)$ is given by
  \begin{align*}
      \left\langle - q^{\frac{n}{2}} \Tr \Theta _C ^{n} \right\rangle_{E_3(g)}
      =& \eta_n  q^{n/2} + \eta_n - \eta_n \sum_{\deg v \divides  \frac{n}{2}}   \frac{(q^{\deg{v}}+1) \deg{v} }{1 + q^{\deg{v}} + q^{2 \deg{v}}}
      \\&
      + \sum_{\deg{v} \mid n}
      \frac{ q^{\deg{v}}  \deg{v} }{1 + q^{\deg{v}} + q^{2 \deg{v}}}
    + \tau_n \sum_{\deg{v} \mid \frac{n}{3}} \frac{ q^{2\deg{v}} \deg{v} }{1 + q^{\deg{v}} + q^{2 \deg{v}}}
    \\&
    + O\left(  q^{-\delta g} q^{(B+1)n }\right),
  \end{align*}
  where the sums are over all places $v$ of $\FF_q(X)$.
\end{theorem}

\begin{proof}
    We rewrite equation~\eqref{eqn:thorne-xiong} as
    \begin{equation*}
        -q^{\frac{n}{2} } \Tr \Theta_C ^n
        =
        \sum_{\alpha} \sum_{\substack{v \in \mathcal{V}_\alpha(C) \\ \deg v \divides  \frac{n}{d_\alpha} }}\delta_\alpha \deg v,
    \end{equation*}
    where $\alpha = 1,2,3,4,5$ indexes the five terms in equation~\eqref{eqn:thorne-xiong}; we also use the index $\alpha$ to refer to the type of ramification associated to the curve $C$ in each term, more precisely as $v \in \mathcal{V}_\alpha(C)$.
    Note that $\delta_1 = \delta_2 = 2$, $\delta_3 = 1$, $\delta_4 = 3$, $\delta_5 = -1$, and $d_1=d_3=d_5=1$, $d_2=2$ and $d_4=3$.

    We now average over our family of curves with genus $g$ to obtain
    \begin{align*}
        \left\langle - q^{\frac{n}{2}} \Tr \Theta_C^n \right\rangle_{E_3(g)}
        &=
        \frac{1}{\# E_3(g)} \sum_{\F_q(C) \in E_3(g)}
        \sum_{\alpha} \sum_{\substack{v \in \mathcal{V}_\alpha(C) \\ \deg v \divides  \frac{n}{d_\alpha} }}\delta_\alpha \deg v \\
        &=
        \sum_{\alpha} 
        \sum_{\deg v \divides  \frac{n}{d_\alpha} }
        \delta_\alpha \deg v
        \frac{\# E_3(g, v,\alpha)}{\# E_3(g)} \\
        &=
        \sum_{\alpha} 
        \sum_{\deg v \divides  \frac{n}{d_\alpha} } \left(
            \delta_\alpha \deg v
            \; c_{v, \alpha}
            +
        O\left(\deg v \, q^{-\delta g}   q^{B \deg v} \right) \right),
    \end{align*}
    where the second equality is obtained by swapping the order of the sums and the third equality follows from Theorem~$\ref{thm:zhao}$.
    Note that the sum of the error terms is
    \begin{equation*}
        O\left(  q^{-\delta g} q^{(B+1)n }\right).
    \end{equation*}
    Writing
    \begin{equation*}
        A(v) = \frac{q^{2 \deg{v}} \deg{v}}{1 + q^{\deg{v}} + q^{2 \deg{v}}},
    \end{equation*}
    we have
    \begin{align*}
        \sum_{\alpha} \sum_{\deg v \divides  \frac{n}{d_{\alpha}} } \delta_{\alpha} \deg v \; c_{v,\alpha}
        &  =
        \eta_n \sum_{\deg v \divides  \frac{n}{2}}   A(v)
 +  \sum_{\deg v \divides  n} A(v) q^{-\deg{v}}
        + \tau_n \sum_{\deg v \divides  \frac{n}{3}} A(v),
    \end{align*}
     and
    \begin{align*}
    \eta_n \sum_{\deg v \divides  \frac{n}{2}}   A(v)
    & =
    \eta_{n} \sum_{\deg v \divides  \frac{n}{2}}  \deg{v} - \eta_n \sum_{\deg v \divides  \frac{n}{2}}   \frac{\left( q^{\deg{v}} + 1 \right) \deg{v} }{1 + q^{\deg{v}} + q^{2 \deg{v}}}\\
    & = \eta_n \left( q^{n/2} + 1 \right) - \eta_n \sum_{\deg v \divides  \frac{n}{2}}   \frac{\left( q^{\deg{v}} + 1 \right) \deg{v} }{1 + q^{\deg{v}} + q^{2 \deg{v}}}
    \end{align*}
    using \eqref{eqn:pnt}.
\end{proof}

\begin{corollary}
    \label{cor:cubic-ngalois}
    For any $\varepsilon > 0$, and as $g \rightarrow \infty$,
    \begin{equation*}
        \left\langle \Tr \Theta_C ^n \right\rangle_{E_3(g)}  = - \eta_n
         + O \left(q^{-\frac{n}{6}} + q^{-\delta g+\left(B+\frac{1}{2}\right) n} \right).
    \end{equation*}
    Further,
    for $6 \log_q{g} < n < \frac{\delta g}{B+1/2}$, and as $g \rightarrow \infty$,
    \begin{equation*}
        \left\langle \Tr \Theta_C ^n \right\rangle_{E_3(g)}
        =  \int_{\USp(2g)} \Tr U^n \dd U + O\left( \frac{1}{g} \right).
    \end{equation*}
\end{corollary}
\begin{proof}
    \emph{Mutatis mutandis} Corollaries~\ref{cor:moment-hyperelliptic}~and~\ref{cor:moment-ellcovers}.
\end{proof}

\begin{theorem}
   \label{thm-oneLD}
   Let $\delta,B > 0$ be fixed constants as in Theorem~\ref{thm:zhao}.
   Let $f$ be an even test function in the Schwartz space $\mathcal{S}(\mathbb{R})$ with $ \operatorname{supp} \hat{f} \subset (- \frac{\delta}{2B + 1}, \frac{\delta}{2B + 1} )$, then for any $\varepsilon > 0$,
   \begin{align*}
      \left \langle W_f (\Theta_C) \right\rangle_{E_3(g)} =  \int_{\USp (2g)} W_f(U) \dd U - \frac{\hat{f}(0)}{g}  \kappa  + O\left( \frac{1}{g^{2 - \varepsilon}} \right),
  \end{align*}
  where
  \begin{equation} \label{def-kappa}
    \begin{split}
      \kappa = \frac{1}{q-1} &-
      \sum_{v}
      \frac{(1+q^{\deg v}) \deg{v} }{ \left( q^{\deg v} - 1 \right) \left( 1 + q^{\deg v} + q^{2 \deg v } \right)} \\
      &+ \sum_{v}
      \frac { q^{\deg v} \deg v }{\left( q^{\deg v /2} - 1 \right)\left( 1 + q^{\deg v} + q^{2 \deg v} \right) } \\
      &+ \sum_{v}
      \frac{ q^{2 \deg v} \deg{v} }{ \left( q^{3 \deg v /2 } - 1 \right)\left( 1 + q^{ \deg v} + q^{2 \deg v} \right)},
    \end{split}
  \end{equation}
where the sums are over all places of $\FF_q(X)$.
Moreover,
\begin{align*}
      \lim_{g \rightarrow \infty} \left \langle W_f (\Theta_C) \right\rangle_{E_3(g)}
      & = \lim_{g \rightarrow \infty} \int_{\USp (2g)} W_f(U) \dd U\\
      & = \int_{\mathbb{R}} f(x) \left( 1 - \frac{\sin (2 \pi x)}{2 \pi x} \right) \dd x.
  \end{align*}
\end{theorem}
\begin{proof}
    Since the function $\hat{f}$ is continuous, its support is contained in $[-\alpha, \alpha]$ for some $0 < \alpha < \frac{\delta}{2B + 1}$.
    Averaging $W_f(\Theta_C)$ over our family of curves using \eqref{formula-for-OLD} and Theorem~\ref{thm:trace-cubicnongalois} we get for $0 < \varepsilon < \delta - 2 \alpha (B+1)$ that
    \begin{equation*}
        \left\langle W_f(\Theta_C) \right\rangle_{E_3(g)} = \hat{f}(0) - \frac{1}{g} \sum_{n=1} ^{\alpha g} \hat{f}\left( \frac{n}{g} \right)
        - \frac{1}{g}  \sum_{n=1}^{ 2 \alpha g} \hat{f}\left( \frac{n}{2g} \right) q^{-n/2} F(n)
        +O\left(q^{-\varepsilon g} \right)
    \end{equation*}
    where
    \begin{align*}
        F(n) & : =  \eta_n - \eta_n \sum_{\deg v \divides  \frac{n}{2}}   \frac{ {(q^{\deg v}+1)} \deg{v} }{{1 + q^{ \deg v}}  + {q^{ 2 \deg v}} }  \\
      &+ \sum_{\deg{v} \divides n}
      \frac{ {q^{\deg v}}   \deg{v} }{{1 + q^{ \deg v}}  + {q^{ 2 \deg v}} }
    + \tau_n \sum_{\deg{v} \divides \frac{n}{3}} \frac{ {q^{2 \deg v}}  \deg{v} }
    {1 + q^{ \deg v}  + q^{ 2 \deg v} }.
  \end{align*}
  Moreover, the two first terms can be rewritten as
  \begin{equation*}
      \int_{\USp (2g)} W_f(U) \dd U = \hat{f}(0) - \frac{1}{g} \sum_{1 \leq n \leq \alpha g} \hat{f} \left( \frac{n}{g} \right),
    \end{equation*}
    using~\eqref{U-moments}.
    Therefore,
    for $0 < \varepsilon < \delta - \alpha(2B+2)$ we have
    \begin{equation*}
         \left\langle W_f(\Theta_C) \right\rangle_{E_3(g)} =  \int_{\USp (2g)} W_f(U) \dd U
         - \frac{1}{g}  \sum_{n=1} ^{2 \alpha g} \hat{f}\left( \frac{n}{2g} \right) q^{-n/2} F(n)
        +O\left(q^{-\varepsilon g} \right).
    \end{equation*}

    We now compute the lower order terms for each of the sums of $F(n)$ as defined above. We have
    \begin{eqnarray*}
        && \sum_{n=1}^{2 \alpha g} \hat{f} \left( \frac{n}{2g} \right) \frac{\eta_n}{q^{n/2}}  \sum_{\deg v \divides  \frac{n}{2}} \frac{ (1 + q^{\deg v}) \deg{v}}{1 + q^{ \deg v}  + q^{ 2 \deg v}}  \\
        && \hspace{1cm} = \sum_{\deg{v} \leq  \alpha g}  \frac{ (1+q^{\deg v}) \deg{v}}{1 + q^{ \deg v}  + q^{ 2 \deg v}}
        \sum_{k \deg{v} \leq \alpha g} \hat{f} \left( \frac{k \deg{v}}{g} \right) \frac{1}{{q^{k \deg v }}};\\
        &&\sum_{n=1}^{2 \alpha g} \hat{f} \left( \frac{n}{2g} \right) \frac{1}{q^{n/2}} \sum_{\deg v \divides  {n}}  \frac{q^{{\deg{v}}} \deg{v}}
        {1 + q^{ \deg v}  + q^{ 2 \deg v}}  =\\
        &&\sum_{\deg{v} \leq 2 \alpha g}   \frac{{q^{\deg v}}  \deg{v}} {1 + q^{ \deg v}  + q^{ 2 \deg v}}
        \sum_{k \deg{v} \leq 2 \alpha g} \hat{f} \left( \frac{k \deg{v}}{2g} \right) \frac{1}{{q^{ k      \deg v /2}}};\\
    &&
      \sum_{n=1}^{2 \alpha g} \hat{f} \left( \frac{n}{2g} \right) \frac{\tau_n}{q^{n/2}} \sum_{\deg v \divides  \frac{n}{3}} \frac{ q^{2\deg{v}} \deg{v}}
      {1 + q^{ \deg v}  + q^{ 2 \deg v}}  = \\
      && \sum_{\deg{v} \leq  2 \alpha g}  \frac{ {q^{2 \deg v}}  \deg{v}} {1 + q^{ \deg v}  + q^{ 2 \deg v} }
        \sum_{3 k \deg{v} \leq 2 \alpha g} \hat{f} \left( \frac{3k \deg{v}}{2g} \right) \frac{1}{ q^{ 3 k \deg v / 2} }.
  \end{eqnarray*}
    As before, we break the range of the inside sum at $g \phi(g)$ where $\phi(g)$ is a function which tends to $0$ as $g$ tends to infinity, and we use the  Taylor expansion for $\hat{f}(x)$ in the first range to get that the first, second and third sum above are respectively
    \begin{gather*}
        \hat{f}(0) \sum_{v}  \frac{ (1+q^{\deg v}) \deg{v} }{ \left( q^{\deg v} - 1 \right) \left( 1 + q^{\deg v} + {q^{2 \deg v}} \right)}
        +
        O\left( \phi(g) + q^{-g \phi(g)} + q^{-2 \alpha g } \right)
        \\
        \hat{f}(0) \sum_{v} \frac{ q^{\deg v}  \deg{v} }{ \left( q^{\deg v / 2} - 1 \right) \left( 1 + q^{\deg v} + q^{2 \deg v}  \right) }
        + O\left(\phi(g) + q^{-g \phi(g)} + q^{-3 \alpha g}\right) \\
        \hat{f}(0) \sum_{v}  \frac{ q^{2 \deg v} \deg{v} } { \left( q^{3 \deg v/2} - 1 \right) \left( 1 + q^{\deg v} +  q^{2 \deg v} \right)}
        + O\left(\phi(g) + q^{-g \phi(g)} + q^{-3 \alpha g} \right),
    \end{gather*}
     and similarly
    $$\sum_{n=1}^{g \phi(g)} \hat{f} \left( \frac{n}{2g} \right) \frac{\eta_n}{q^{n/2}}  = \frac{\hat{f}(0)}{q-1} + O\left(\phi(g) + q^{-g \phi(g)}\right).
    $$
    For the remaining range from $g \phi(g)$ to $2 \alpha g$, working as in the proofs of Theorems~\ref{thm:1level-density-hyperelliptic} and~\ref{thm:1level-density-ell}, we have that each of the four sums is $O\left(\alpha g q^{-g \phi(g)}\right)$.
     By choosing $\phi(g) = g^{-1+\varepsilon}$, we get that
 \begin{equation*}
 - \frac{1}{g} \sum_{n=1}^{2 \alpha g} \hat{f} \left( \frac{n}{2g} \right) {q^{-n/2}} F(n)
= - \frac{1}{g} \hat{f}(0) \kappa + O \left( \frac{1}{g^{2-\varepsilon}} \right),
\end{equation*}
which proves the first statement.
  Taking the limit $g \rightarrow \infty$ we get the second part of the theorem.
\end{proof}

We now compare the results of the above theorem with the results obtained by  Yang for the one-level density of cubic non-Galois extensions over number fields~\cite{yang2009distribution}.
Yang's results hold for $\operatorname{supp} \hat{f} \subset (-c, c)$, where
\begin{equation*}
    c = \frac{2(1-A)}{2B+1},
\end{equation*}
and the parameters $0<A<1$ and $B>0$ are such that
\begin{equation} \label{ET-yang}
N_{p}(X,T) = c_{P,T} X  + O \left( X^A p^B \right),
\end{equation}
where $N_p(X,T)$ is the number of cubic non-Galois extensions of $\QQ$ with discriminant between $0$ and $X$ and such that the splitting behavior at the prime $p$ is of type $T$, see \cite[Proposition 2.2.4]{yang2009distribution}.
In order to compare it with Theorem~\ref{thm-oneLD}, we need to find the correspondence between the $A$ of \eqref{ET-yang} and the $\delta$ of Theorem~\ref{thm:zhao} (the $B$'s are the same).
We rewrite \eqref{ET-yang} by dividing by the main term given by $N(X)$, the number of non-Galois cubic fields of discriminant up to $X$ which is $C X$ for some absolute constant $C$, and we rewrite \eqref{ET-yang} as
\begin{eqnarray} \label{ET-yang-2}
\frac{N_{p}(X,T)}{N(X)} = c'_{P,T}  + O \left( X^{A-1} p^B \right).
\end{eqnarray}
In the situation of Theorem~\ref{thm:zhao},
since
$\# E_3(g) \sim q^{2g+4}$, we have for one place $v$ that
\begin{eqnarray} \label{ET-zhao-2}
\frac{\# E_3(g, \{ v \}, \Omega)}{\# E_3(g)} = c_{v}  + O \left( \left( q^{2g} \right)^{-\delta/2} q^{B \deg v} \right).
\end{eqnarray}
Then, to compare \eqref{ET-yang-2} and \eqref{ET-zhao-2}, we set
$$A-1 = -\frac{\delta}{2} \iff \delta=2-2A.$$
Then, we have that the support of the Fourier transform in Theorem~\ref{thm-oneLD} is $(-c, c)$ where
$$c =\frac{\delta}{2B+1} = \frac{2-2A}{2B+1},$$
which agrees with the support of the Fourier transform in~\cite[Proposition 2.2.4]{yang2009distribution}.

\section{Explicit error terms and the Lindel\"{o}f bound}
\label{sec:lindelof}

In this section we explain our approach to make the dependence on the place $v_0$ explicit in Theorems~\ref{thm:density-hyperelliptic} and~\ref{thm:general-ell}.
We start by reviewing how the counting of function field extensions ramifying (or splitting or inert) at a given finite place $v_0$ is obtained in~\cite{BDFKLOW}, and how the dependence on the place $v_0$ reduces to obtaining the Lindel\"{o}f bound for the Dirichlet L-functions $L(1/2 + it, \chi)$, where $\chi$ is a Dirichlet character of modulus $v_0$ and order $\ell$.
We conclude this section by proving this bound.

The counting of function fields extensions in~\cite{BDFKLOW} is done by writing explicitly the generating series for the extensions, and applying the Tauberian Theorem~\cite[Theorem 2.5]{BDFKLOW} to the generating series.
As usual, this involves moving the line of integration and applying Cauchy's residue theorem to the relevant region.
The main term will be given by the sum of the residues at the poles in the region, and this is where the main terms of Theorems~\ref{thm:density-hyperelliptic} and~\ref{thm:general-ell} come from.
The error term comes from evaluating the integral at the limit of the region of analytic continuation of the generating series, which involves bounding the generating series on some half line.

We start by looking at the counting for cyclic extensions of degree $\ell$ with conductor of degree $d$ which ramify (or not ramify) at a given place $v_0$.
In this case, the generating series $\mathcal{F}_R(s)$ and $\mathcal{F}_U(s)$, respectively, converge absolutely for $\re(s) > \frac{1}{\ell-1}$ with a pole of order $\ell-1$ at $s = \frac{1}{\ell-1}$, which gives the main term.
Each generating series has analytic continuation to $\re(s) = \frac{1}{2(\ell-1)} + \varepsilon$ for any $\varepsilon > 0$, and the error term is then bounded by
\begin{equation*}
    O \left( q^{(\frac{1}{2} + \varepsilon)d} M \right),
\end{equation*}
where $M$ is the maximum value taken by $\mathcal{F}_R(s)$ (or  $\mathcal{F}_U(s)$) on the line $\re(s) = \frac{1}{2(\ell-1)} + \varepsilon$.
It is important to note that the generating series are absolutely bounded on this line, i.e., the bound does not depend on $v_0$, but might depend on $q$ and $\ell$, and the results of Theorems~\ref{thm:density-hyperelliptic} and~\ref{thm:general-ell} follow.
There is a difference between the case $\ell = 2$ and $\ell \geq 3$, as the generating series is written as the sum of two functions, one with a pole of order $\ell-1$, and one with poles of order $1$. If $\ell \geq 3$, the main term comes from the pole of order $\ell-1$ only. If $\ell = 2$, the two poles are simple and there is some cancellation between contributions of the residues at the two poles.

It remains to deal with the error terms for the two unramified cases, namely counting extensions split at $v_0$ and inert at $v_0$.
In this case the argument works the same way for all $\ell \geq 2$. Let $\xi_\ell$ be a primitive $\ell$-th root of unity, and let
\begin{equation*}
    \chi_{v, \ell}(v_0) = \left( \frac{v_0}{v} \right)_\ell,
\end{equation*}
be the $\ell$-th power residue symbol,
which is a Dirichlet character of order $\ell$ and modulus $v_0$ over $\F_q(X)$.

The generating series for $E(\mathbb{Z}/\ell \mathbb{Z}, \ell, v_0, \mathrm{split})$ is
\begin{equation}
    \label{SversusU}
    \mathcal{F}_S(s)= \frac{1}{\ell} \mathcal{F}_U(s) +\frac{1}{\ell^2} \sum_{j=0}^{\ell-1} \sum_{k=1}^{\ell-1} \left(\sum_{r=0}^{\ell - 1} \xi_\ell^{-rk \deg{v_0}}\right) \mathcal{M}_{j,k} (s,v_0,\mathrm{split}),
\end{equation}
where $\mathcal{M}_{j,k} (s,v_0,\mathrm{split})$ is given by
\begin{equation*}
    \prod_{v \neq v_0} \left( 1 + \left( \xi_\ell^{j \deg v} \chi_{v, \ell}(v_0)^k + \dots + \xi_\ell^{(\ell-1)j \deg{v}} \chi_{v, \ell}(v_0)^{(\ell-1) k} \right) Nv^{-(\ell-1)s} \right).
\end{equation*}
As before, the count is then obtained by applying the Tauberian theorem to the generating series $\mathcal{F}_S(s)$.
This series converges absolutely for $\re(s) > \frac{1}{\ell-1}$ with a pole of order $\ell-1$ at $s = \frac{1}{\ell-1}$, which gives the main term.
The function $\mathcal{F}_S(s)$ has analytic continuation to $\re(s) = \frac{1}{2(\ell-1)} + \varepsilon$ for any $\varepsilon > 0$.
The error term is bounded by
\begin{equation*}
    O \left( q^{\frac{1}{2} + \varepsilon} M \right),
\end{equation*}
where $M$ is the maximal value of $\mathcal{F}_S(s)$ for $\re(s) = \frac{1}{2(\ell-1)} + \varepsilon.$
As we mentioned above, $\mathcal{F}_U(s)$ is absolutely bounded on this line, thus we have to bound $\mathcal{F}_S(s) -\frac{1}{\ell} \mathcal{F}_U(s)$ on the aforementioned line.
We can rewrite $\mathcal{M}_{j,k} (s,v_0,\mathrm{split})$ as
\begin{equation*}
    \mathcal{G}(s)
    \prod_{r=1}^{\ell-1}
    \prod_{v \neq v_0} \left( 1 -  \xi_\ell^{ r j \deg v} \chi_{v, \ell}(v_0)^{rk} N(v)^{-(\ell-1)s} \right)^{-1},
\end{equation*}
where the function $\mathcal{G}(s)$ converges absolutely for $\re(s) > \frac{1}{2(\ell-1)} + \varepsilon$, and it is uniformly bounded in that region.
Hence our task is reduced to bounding the L-functions
\begin{equation*}
    L_{j,k}(s) =  \prod_{v \neq v_0} \left( 1 -  \frac{ \xi_\ell^{j \deg v} \chi_{v, \ell}(v_0)^k }{ N(v)^{(\ell-1)s} } \right)^{-1},
\end{equation*}
on the line $\re(s) = \frac{1}{2(\ell-1)} + \varepsilon$. The $L_{j,k}(s)$ ($0\leq j,k\leq \ell -1$)
are Dirichlet L-functions associated to some character $\chi$ of modulus $v_0$ and order $\ell$ and we need to evaluate them at $s = 1/2 + \varepsilon + i t$.
Indeed, if $\xi$ be any root of unity and we write $\xi = q^{- i \theta}$, then we have
\begin{equation*}
    \prod_{v} \left( 1 - \frac{\chi(v) \xi^{\deg{v}}}{N(v)^s} \right)^{-1} = \prod_{v} \left( 1 - \chi(v) q^{-(s+ i \theta) \deg{v}} \right)^{-1} = L(s + i \theta, \chi).
\end{equation*}

In the following theorem we prove that the Lindel\"{o}f Hypothesis is true for the L-functions $L(s, \chi)$ associated with non-trivial Dirichlet characters of $\F_q[X]$.
There are two main ingredients in our proof, the Riemann Hypothesis and~\cite[Theorem 8.1]{carneiro2010some}, an Erd\"{o}s--Tur\'{a}n-type inequality, proved by Carneiro and Vaaler, bounding the size of polynomials inside the unit circle.
This approach was suggested to us by Soundararajan who used the same approach in a paper in collaboration with Chandee~\cite{chandee-sound} to get similar bounds for $\zeta(\frac{1}{2} + it)$. We are very thankful for his suggestion and his help.
There are other bounds in the literature for $\log \left| L(1/2+it, \chi_{v_0})  \right|$, for example the bound proved by Altung and Tsimerman in \cite[p.45]{altug-tsimerman}
$$
\log \left| L(1/2, \chi_{v_0})  \right| \leq \frac{2g}{\log_q{(g)}} + 4q^{1/2} g^{1/2},
$$
where $q$ is prime. Then, the bound below improves the constant from 1 to the optimal constant $\frac{\log{2}}{2}$ (we recall that $d = 2g + 2$ for hyperelliptic curves).
Very  recently, similar bounds with the constant $\frac{\log{2}}{2}$ were obtained by Florea \cite[Corollary 8.2]{florea-four}  using a different proof based in similar ideas, inspired by the work of Carneiro and Chandee \cite{CC}.  Her proof also allows her to get better bounds for $\log \left| L(\alpha+it, \chi_{v_0})  \right|$ for $\alpha \geq 1/2$ (the L-function gets smaller as one moves away the critical line).

\begin{theorem}\label{thm:lh}
  Let $v_0$ be a finite place of $\FF_q(X)$ and denote by $\chi_{v_0}$ be the $\ell$-th power residue symbol, which is a Dirichlet character of modulus $v_0$.
Let $d$ be the degree of the conductor of the character, and
    let $L(s, \chi_{v_0})$ be the L-function attached to $\chi_{v_0}$.
    For any $s = \sigma + it$ with $\sigma \geq 1/2$, we have as $d \rightarrow \infty$,
    \begin{equation} \label{L1}
        \log \left| L(s, \chi_{v_0})  \right| \leq \left( \frac{\log 2}{2} + o(1) \right) \frac{d}{\log_q{d}}.
    \end{equation}
   Hence, for any  $\varepsilon > 0$, we have
    \begin{equation} \label{L2}
        L(s,  \chi_{v_0}) \ll_{q, \varepsilon} \left( q^d \right)^\varepsilon.
    \end{equation}
\end{theorem}

\begin{proof}
    Let $v_0(X)$ be the polynomial of $\F_q[X]$ corresponding to the place $v_0$, and consider the curve
    \begin{equation*}
        C_{v_0} : Y^\ell = v_0(X).
    \end{equation*}
    Let $g$ be the genus of the curve. Then, $d-2 = 2g/(\ell-1)$, and the zeta function of the curve $C_{v_0}$ writes as
    \begin{equation*}
        Z_{C_{v_0}}(u) = \frac{\prod_{j=1}^{2g} \left( 1 - u e^{i \theta_{j}} \right)}{(1-u)(1-qu)} =
        \frac{\prod_{k=1}^{\ell-1} L(u, \chi_{v_0}^k)}{(1-u)(1-qu)}
        ,
    \end{equation*}
    where
    \begin{equation*}
        L(u, \chi_{v_0}^k) = \prod_{j=1}^{2g/(\ell-1)}
       \left( 1 - \sqrt{q} e^{i\theta_{k,j}} u\right),
    \end{equation*}
    renaming the roots of $Z_{C_{v_0}}(u)$.

    Without loss of generality, take $k=1$ and rewrite
 \begin{equation*}
        L(u, \chi_{v_0}) = \prod_{j=1}^{d-2}
       \left( 1 - \sqrt{q} e^{i\theta_{j}} u\right).
    \end{equation*}
    Evaluating at $u=q^{-s}$ for $s=\sigma+it$ with $\sigma \geq 1/2$, we have
    \begin{equation}
        \label{eqn:lchi}
        L(s, \chi_{v_0}) = \prod_{j=1}^{d-2} \left( 1 - e^{i (\theta_j- t\log{q})} q^{1/2-\sigma} \right).
    \end{equation}
    We consider the polynomial
    \begin{equation*}
        F(z) = \prod_{j = 1}^{d-2} \left(z -  e^{i (\theta_j- t\log{q})} q^{1/2-\sigma} \right),
    \end{equation*}
    and we notice that all $\alpha_j = q^{1/2-\sigma} e^{i (\theta_j- t\log{q})}$ are such that $|\alpha_j| \leq 1$ since $\sigma \geq 1/2$.
    We now use~\cite[Theorem 8.1]{carneiro2010some} which says that for
    \begin{equation*}
        F_M(z)= \prod_{m=1}^M (z - \alpha_m)
    \end{equation*}
    where $| \alpha_m | \leq 1$ for $1 \leq m \leq M$, we have  for any positive integer $N$ that
    \begin{align} \label{CV}
        \sup_{|z| \leq 1} \log | F_M(z) | \leq \log{2} \frac{M}{N+1}  + \sum_{n=1}^N \frac{1}{n} \left| \sum_{m=1}^M \alpha_m^n \right|.
    \end{align}

    We then have to evaluate the sums of powers
    \begin{equation*}
    \sum_{j=1}^{d-2} \alpha_j^n \ll \sum_{j=1}^{d-2} e^{i n \theta_j}.\end{equation*}

    Taking the logarithm derivative on both sides of~\eqref{eqn:lchi} (similarly to the
    proofs of Lemma~\ref{lemma:hyperelliptic} and~\ref{lemma:traces-ell}), we derive the following identity for $n \geq 1$
    \begin{equation*}
        \sum_{j=1}^{d-2} e^{i n \theta_{j}}
        =
        - q^{-\frac{n}{2}} \sum_{\deg v \divides  n} \deg v \left(\chi_{v_0}  (v) \right)^{\frac{n}{\deg v}},
    \end{equation*}
    where the sum is over all places $v$ of $\FF_q(X)$ (including the place at infinity).
    Therefore, using\eqref{eqn:pnt},
    \begin{align*}
        \left |  \sum_{j=1}^{d-2} e^{i n \theta_{k,j}} \right|
        & = q^{-\frac{n}{2}} \left( 1 + q^n \right) \leq 1 + q^{\frac{n}{2}}
    \end{align*}

    Replacing the bound above in~\eqref{CV} with $M = 2g/(\ell-1) = d-2$, we get
    \begin{align*}
        \sup_{|z| \leq 1} \log | F(z) |
        &\leq \log{2} \frac{d-2}{N + 1}  + \sum_{n=1}^{N} \frac{1+q^{\frac{n}{2}}}{n}\\
        & \leq \log{2} \frac{d-2}{N + 1} + \log 2 N + 2 \frac{q^{\frac{N}{2}}}{N}
    \end{align*}
    The theorem follows by taking $N = \lfloor \left(2 - f(d) \right) \log_q d \rfloor$, where $f(d)$ is any positive function $f(d)$ such that $f(d) = {o}(1)$ and $e^{- f(d) \log_q{d}} = {o}(1)$, for example $f(d) = \frac{\log_q{\log_q{d}}}{\log_q{d}}$.
    Without loss of generality assume $N > 0$, and we have
    \begin{align*}
        \sup_{|z| \leq 1} \log | F(z) |  &\leq  \log{2} \; \frac{d}{(2 - f(d)) \log_q d} + {o} \left( \frac{d}{\log_q{d}} \right) \\
        & \leq \left( \frac{\log 2}{2} + {o}(1) \right) \frac{d}{\log_q d},
    \end{align*}
    which shows \eqref{L1}, and \eqref{L2} follows.
\end{proof}

\bibliographystyle{alpha}
\bibliography{biblio}

\newcommand{\etalchar}[1]{$^{#1}$}
\begin{thebibliography}{BDF{\etalchar{+}}16}

\bibitem[AT14]{altug-tsimerman}
Salim~Ali Altug and Jacob Tsimerman.
\newblock Metaplectic ramanujan conjecture over function fields with
  applications to quadratic forms.
\newblock {\em International Mathematics Research Notices}, No. 13:3465--3558,
  2014.

\bibitem[BDF{\etalchar{+}}16]{BDFKLOW}
Alina Bucur, Chantal David, Brooke Feigon, Nathan Kaplan, Matilde Lalin, Ekin
  Ozman, and Melanie Wood.
\newblock The distribution of points on cyclic covers of genus g.
\newblock {\em International Mathematics Research Notices}, No. 14:4297--4340,
  2016.

\bibitem[BF14]{Bui-Florea}
H.~M. Bui and Alexandra Florea.
\newblock Zeros of quadratic dirichlet l-functions in the hyperelliptic
  ensemble.
\newblock {\em Preprint}, 2014.
\newblock {\tt arXiv:1605.07092}.

\bibitem[BG01]{bg}
Bradley~W. Brock and Andrew Granville.
\newblock More points than expected on curves over finite field extensions.
\newblock {\em Finite Fields Appl.}, 7(1):70--91, 2001.
\newblock Dedicated to Professor Chao Ko on the occasion of his 90th birthday.

\bibitem[CC11]{CC}
Emmanuel Carneiro and Vorropan Chandee.
\newblock Bounding $\zeta(s)$ in the critical strip.
\newblock {\em Journal of Number Theory}, 131:363--384, 2011.

\bibitem[Chi15]{chinis}
Iakovos~Jake Chinis.
\newblock Traces of high powers of the frobenius class in the moduli space of
  hyperelliptic curves.
\newblock {\em Preprint}, 2015.
\newblock {\tt arXiv:1510.06350}.

\bibitem[CS11]{chandee-sound}
Vorrapan Chandee and K.~Soundararajan.
\newblock Bounding {$\vert \zeta(\frac12+it)\vert $} on the {R}iemann
  hypothesis.
\newblock {\em Bull. Lond. Math. Soc.}, 43(2):243--250, 2011.

\bibitem[CV10]{carneiro2010some}
Emanuel Carneiro and Jeffrey Vaaler.
\newblock Some extremal functions in fourier analysis. {II}.
\newblock {\em Transactions of the American Mathematical Society},
  362(11):5803--5843, 2010.

\bibitem[DS94]{eigenvalues}
Persi Diaconis and Mehrdad Shahshahani.
\newblock On the eigenvalues of random matrices.
\newblock {\em J. Appl. Probab.}, 31A:49--62, 1994.
\newblock Studies in applied probability.

\bibitem[DW88]{datskovsky1988density}
Boris Datskovsky and David~J Wright.
\newblock Density of discriminants of cubic extensions.
\newblock {\em J. reine angew. Math}, 386:116--138, 1988.

\bibitem[Flo16]{florea-four}
Alexandra Florea.
\newblock The fourth moments of quadratic dirichelt l-functions over function
  fields.
\newblock {\em Preprint}, 2016.
\newblock {\tt arXiv:1609.01262}.

\bibitem[FPS16]{FPS}
Daniel Fiorilli, James Parks, and Anders S\"{o}dergren.
\newblock Traces of high powers of the frobenius class in the moduli space of
  hyperelliptic curves.
\newblock {\em Preprint}, 2016.
\newblock {\tt arXiv:1601.06833}.

\bibitem[FR10]{fr}
Dmitry Faifman and Ze{\'e}v Rudnick.
\newblock Statistics of the zeros of zeta functions in families of
  hyperelliptic curves over a finite field.
\newblock {\em Compos. Math.}, 146(1):81--101, 2010.

\bibitem[KR09]{KURU}
P{\"a}r Kurlberg and Ze{\'e}v Rudnick.
\newblock The fluctuations in the number of points on a hyperelliptic curve
  over a finite field.
\newblock {\em J. Number Theory}, 129(3):580--587, 2009.

\bibitem[KS99]{katz-sarnak}
Nicholas~M. Katz and Peter Sarnak.
\newblock {\em Random matrices, {F}robenius eigenvalues, and monodromy},
  volume~45 of {\em American Mathematical Society Colloquium Publications}.
\newblock American Mathematical Society, Providence, RI, 1999.

\bibitem[Ros02]{rosen2002number}
M.~Rosen.
\newblock {\em Number Theory in Function Fields}.
\newblock Graduate Texts in Mathematics. Springer, 2002.

\bibitem[Rud10]{rudnick}
Ze{\'e}v Rudnick.
\newblock Traces of high powers of the {F}robenius class in the hyperelliptic
  ensemble.
\newblock {\em Acta Arith.}, 143(1):81--99, 2010.

\bibitem[TX14]{thorne-xiong}
Frank Thorne and Maosheng Xiong.
\newblock Distribution of zeta zeroes for cyclic trigonal curves over a finite
  field.
\newblock {\em Preprint}, 2014.

\bibitem[Yan09]{yang2009distribution}
Andrew Yang.
\newblock {\em Distribution problems associated to zeta functions and invariant
  theory}.
\newblock Princeton University, 2009.
\newblock PhD Thesis under Peter Sarnak.

\bibitem[Zha]{Zhao}
Yongqiang Zhao.
\newblock On sieve methods for varieties over finite fields.
\newblock {\em Preprint}.

\end{thebibliography}

\end{document}